\documentclass{article}

\usepackage[margin=1in]{geometry}
\usepackage{amsmath}
\usepackage{amssymb}
\usepackage{amsthm}
\usepackage{booktabs}
\usepackage{graphicx}
\usepackage{mathtools}
\usepackage{accents}
\usepackage[round]{natbib}
\usepackage{url}
\usepackage{multirow}

\theoremstyle{plain}
\newtheorem{theorem}{Theorem}
\newtheorem{proposition}{Proposition}
\newtheorem{lemma}{Lemma}

\newtheoremstyle{newremark}
{\topsep}   
{\topsep}   
{\upshape}  
{0pt}       
{\itshape}  
{.}         
{5pt plus 1pt minus 1pt} 
{\thmname{#1}\thmnumber{ \itshape#2}\thmnote{ (#3)}}
\theoremstyle{newremark}

\newtheorem{assumption}{Assumption}
\newtheorem{remark}{Remark}

\newcommand\munderbar[1]{\underaccent{\bar}{#1}}

\newcommand{\R}{\mathbb{R}}

\title{Estimation and convergence rates in the \\ distributional single index model}
\author{Fadoua Balabdaoui\footnote{Corresponding author, email: fadouab@ethz.ch}, Alexander Henzi, and Lukas Looser \\[0.5em]
Seminar for Statistics, ETHZ, R\"amistrasse 101, 8092, Z\"urich} 
\date{\today}

\begin{document}
\maketitle

\begin{abstract}
The distributional single index model is a semiparametric regression model in which the conditional distribution functions $P(Y \leq y | X = x) = F_0(\theta_0(x), y)$ of a real-valued outcome variable $Y$ depend on $d$-dimensional covariates $X$ through a univariate, parametric index function $\theta_0(x)$, and increase stochastically as $\theta_0(x)$ increases. We propose least squares approaches for the joint estimation of $\theta_0$ and $F_0$ in the important case where $\theta_0(x) = \alpha_0^{\top}x$ and obtain convergence rates of $n^{-1/3}$, thereby improving an existing result that gives a rate of $n^{-1/6}$. A simulation study indicates that the convergence rate for the estimation of $\alpha_0$ might be faster.  Furthermore, we illustrate our methods in an application on house price data that demonstrates the advantages of shape restrictions in single index models. \\[0.5em]
\textbf{Keywords} \ monotone regression, isotonic distributional regression, single index model
\end{abstract}

\section{Introduction} \label{sec:introduction}
Consider the classical regression framework in which one aims to predict a response variable $Y \in \mathbb{R}$ with covariates $X \in \mathcal{X} \subseteq \mathbb{R}^d$. The popular generalized linear models (GLMs) assume that
\[
    \mathbb{E}[Y|X=x] = g_{\phi}(\alpha_0^{\top}x),
\]
where $Y$ follows an exponential family distribution, $\alpha_0$ is unknown, and $g_{\phi}$ is a monotone transformation known up to a dispersion parameter $\phi$ that does not depend on the covariates. \citet{BalabdaouiDurotJankowski2019} study a semiparametric variant of this model, the monotone single index model, where the function $g_{\phi}$ is replaced by an unknown monotone function $\psi_0$ that is estimated nonparametrically, jointly with $\alpha_0$. The focus of this article is an extension of the monotone single index model introduced by \citet{HenziKlegerZiegel2023}, called the distributional single index model, which aims at estimating conditional cumulative distribution functions (CDFs) of $Y$ given $X$ rather than only its conditional expectation. The model assumes that
\begin{equation} \label{eq:mdsim}
    \mathbb{P}(Y \leq y \mid X = x) = F_0(\theta_0(x), y),
\end{equation}
where $y \mapsto F_0(z, y)$ is an unknown conditional distribution function for all fixed $z \in \mathbb{R}$, $\theta_0 \colon \mathbb{R}^d \rightarrow \mathbb{R}$ a mapping of the $d$-dimensional covariates to $\mathbb{R}$, and monotonicity of $\psi_0$ is replaced by the assumption of stochastic monotonicity. Stochastic monotonicity means that $F_0(z, y)$ is non-increasing in $z$ for all fixed $y \in \mathbb{R}$, so graphically, the conditional CDFs $F_0(z,y)$ shift to the right as $z$ increases, or in simple words, $Y$ tends to attain larger values when $\theta_0(X)$ is large. In this article, we are interested in the special case where $\theta_0(x) = \alpha_0^{\top}x$ is a linear function. The most popular families in generalized linear models --- Gaussian, Binomial, Poisson, Gamma, Inverse Gaussian --- satisfy the stochastic monotonicity assumption of the distributional single index model, save for a change of sign of $\alpha_0$ for decreasing link functions. Thus, the model can be regarded as a semiparametric, distributional extension of GLMs. If $Y$ has finite expectation, then
\[
    \mathbb{E}[Y | X = x] = \int_{0}^{\infty} \left(1-F_0(\alpha_0^{\top}x, y)\right) \, dy - \int_{-\infty}^0 F_0(\alpha_0^{\top}x, y) \, dy
\]
is increasing in $\alpha_0^{\top}x$, so the assumption of stochastic monotonicity is stronger than monotonicity of the conditional expectation in this case. When $Y$ is binary, the distributional single index model becomes a special case of the monotone single index model. Both the monotone single index model and the distributional single index model build on the idea of single index model introduced by \citet{HardleHallIchimura1993}, and we refer the interested readers to the literature reviews in \citet{BalabdaouiDurotJankowski2019} and \citet{HenziKlegerZiegel2023} for a comprehensive discussion of related work.

Rates for the estimation of the conditional CDFs in the distributional single index model have already been obtained by \citet{HenziKlegerZiegel2023}. They showed that for an independent and identically distributed (i.i.d.)\ sample $(X_1, Y_1), \dots, (X_n, Y_n)$ from model \eqref{eq:mdsim}, if $\hat{\theta}_n$ is a uniformly consistent estimator for $\theta_0$ converging at a rate of $o_p((\log(n)/n)^{1/2})$ and if $\hat{F}_n$ is computed on the data $(\hat{\theta}_n(X_1), Y_1), \dots, (\hat{\theta}_n(X_n), Y_n)$ with isotonic distributional regression \citep{MoeschingDuembgen2020, HenziZiegelGneiting2021}, then
\begin{equation} \label{eq:dim_old_rate}
    \sup_{y \in \mathbb{R}, \, x \in \mathcal{X}_{\varepsilon_n}} |\hat{F}_n(\hat{\theta}_n(x), y) - F_0(\theta_0(x), y)| = o_p((\log(n)/n)^{1/6})
\end{equation}
under certain regularity conditions. Here $\mathcal{X}_{\varepsilon} = \{x \in \mathcal{X}\colon \theta_0(x) \pm \varepsilon_n \in I\}$ for an interval $I$ on which $\theta(X)$ has density bounded away from zero and infinity, and $\varepsilon_n > 0$ is a certain sequence converging to zero. When $\hat{\theta}_n$ and $\hat{F}_n$ are computed on independent samples, a faster rate of $o_p((\log(n)/n)^{1/3})$ is achieved, if $\hat{\theta}_n$ converges to $\theta_0$ at least at this rate. \citet{HenziKlegerZiegel2023} provide no theoretical results on the estimation of the index function, and the rate of $o_p((\log(n)/n)^{1/6})$ is likely to be suboptimal, because if $\theta_0 = \hat{\theta}_n$ it should be $o_p((\log(n)/n)^{1/3})$ by Theorem 3.3 of \citet{MoeschingDuembgen2020}, or if $Y$ is binary the results of \citet{BalabdaouiDurotJankowski2019} yield $O_p(n^{-1/3})$ for the estimation of $F_0$ and $\theta_0$ when the latter is linear.

In this article, we focus on the linear case $\theta_0(x) = \alpha_0^{\top}x$, and propose to estimate $(F_0, \alpha_0)$ by minimizing weighted least squares criteria of the form
\begin{equation} \label{eq:weighted_lsc}
	L_n(Q; F, \alpha) = \frac{1}{n} \sum_{i=1}^n \int_{\mathbb{R}} (1\{Y_i \leq t\} - F(\alpha^{\top}X_i, t))^2 \, dQ(t),
\end{equation}
where $Q$ is a Borel measure. We obtain a rate of $O_P(n^{-1/3})$ when $Q$ has a finite support or it is compactly supported Lebesgue continuous with a  bounded density. Furthermore, we investigate an approach with $Q$ equal to the empirical distribution of $Y_1, \dots, Y_n$, which has favorable invariance properties under transformations of the response variable, but the consistency and convergence rates of which remain an open challenge.

The article is structured as follows. In Section \ref{sec:estimation} we describe the estimation method in detail. Convergence rates are derived in Section \ref{sec:rates}. In Section \ref{sec:transformation_invariant}, we present the invariance property result which holds when $Q$ is taken to be the empirical distribution function of the responses. In Section \ref{sec:empirical} we discuss computational aspects and present a simulation study and an application on house price data. We conclude with a discussion in Section \ref{sec:discussion}, and the proofs are deferred to Section \ref{sec:proofs}. Throughout the article, we denote the joint distribution of $(X,Y)$ by $\mathbb{P}$, the marginals by $\mathbb{P}^X$ and $\mathbb{P}^Y$, and the conditional distributions by $\mathbb{P}^{Y|X=x}$ and $\mathbb{P}^{X|Y=y}$, respectively. The empirical distributions of $n$ independent observations are denoted by $\mathbb{P}_n$, $\mathbb{P}_n^X$, $\mathbb{P}_n^Y$. We denote by $\mathrm{supp}(P)$ the support of a probability measure $P$, and by $A^{\circ}$ the interior of a set $A$. The expectation operator $\mathbb{E}[\cdot]$ is understood to be with respect to $\mathbb{P}$, unless explicitly defined differently.

\section{Estimation} \label{sec:estimation}
Let $(X_1, Y_1), \dots, (X_n, Y_n)$ be a sample of covariates and response variable from model $\eqref{eq:mdsim}$, where from now on we always assume that $\theta_0(x) = \alpha_0^{\top}x$. Define $\mathcal{C}_{\alpha} = \{\alpha^{\top}x\colon x \in \mathcal{X}\}$, and let $\mathcal{F}_{\alpha}\colon \mathcal{C}_{\alpha} \times \mathbb{R} \rightarrow [0,1]$ be the class of bivariate functions $F$ for which $y \mapsto F(z, y)$ is a CDF for all fixed $z \in \mathbb{R}$, and $z \mapsto F(z,y)$ is non-increasing for all fixed $y \in \mathbb{R}$. The function $F_0$ and the parameter $\alpha_0$ in \eqref{eq:mdsim} are not identifiable, since $\check{F}_0(z, y) = F_0(z/c, y)$ and $\check{\alpha}_0 = c\cdot \alpha_0$ for $c > 0$ yield the same conditional distributions. Hence, we assume that $\alpha_0 \in \mathcal{S}_{d-1} = \{x \in \mathbb{R}^d\colon \|x\| = 1\}$, and define the class of candidate functions for estimation by
\[
	\mathcal{F} = \{(F,\alpha)\colon \alpha \in \mathcal{S}_{d-1}, \, F \in \mathcal{F}_{\alpha}\}.
\]
To estimate $(F_0, \alpha_0)$, we propose to minimize the least squares criteria of the form given in \eqref{eq:weighted_lsc}. The following proposition describes the solutions of this minimization problem.

\begin{proposition} Assume that $Q$ is locally finite.
\begin{itemize}
	\item[(i)] For a fixed $\alpha  \in \mathcal{S}_{d-1}$, let $z_1 < \dots < z_m$ be the distinct values of $\alpha^{\top}X_1, \dots, \alpha^{\top}X_n$, with multiplicities $n_1, \dots, n_m$. The minimizer of $L_n(Q; F, \alpha)$ in $F$ is uniquely defined in the first argument on $\{z_1, \dots, z_m\}$ and in the second argument on $\mathrm{supp}(Q)$, and it is given by
	\begin{equation} \label{eq:optimal_F}
		\hat{F}_{n, \alpha}(z_i, y) = \min_{k=1,\dots, i} \max_{l=i,\dots,m} \left(\frac{1}{n_k + \dots + n_l} \sum_{j=k}^{l} \, \sum_{s: \,\alpha^{\top}X_s = z_j} 1\{Y_s \leq y\}\right), \ i = 1, \dots, m.
	\end{equation}
	\item[(ii)] Let $S^X = \{\alpha \in \mathcal{S}_{d-1}\colon \alpha^{\top}X_i \neq \alpha^{\top}X_j, \, i,j=1,\dots,n,\, i\neq j\}$. The minimum of $L_n(Q; F, \alpha)$ is achieved for a pair $( \hat{F}_{n, \hat{\alpha}_n}, \hat{\alpha}_n)$ with $\hat{\alpha}_n \in S^X$ and $ \hat{F}_{n, \hat{\alpha}_n}$ given by \eqref{eq:optimal_F}. The minimizer is not unique.
\end{itemize}
\end{proposition}
The estimator $\hat{F}_{n, \alpha}$ in \eqref{eq:optimal_F} is called the isotonic distributional regression in \citet{HenziZiegelGneiting2021}, and the fact that it is a minimizer is due to Theorem 2.1 of that article; the condition that $Q$ is locally finite is only necessary to ensure uniqueness in part (i). It follows directly from \eqref{eq:optimal_F} that $ y \mapsto \hat{F}_{n, \alpha}(\alpha^{\top}X_i, y)$ is indeed a CDF for $i = 1, \dots, m$. For a fixed $\alpha$, the estimator $\hat{F}_{n, \alpha}$ depends on $Q$ only through its support, as can be seen from \eqref{eq:optimal_F}. It suffices to compute it at the distinct values $y_1 < \dots < y_k$ of $Y_1, \dots, Y_n$, since for $y \geq y_1$,
\[
	1\{Y_i \leq y\} = 1\{Y_i \leq y_{l(y)}\}, \ i = 1, \dots, n, \quad \text{with} \quad l(y) = \max\{j \in  \{1, \dots, k\}\colon y_j \leq y\},
\]
and $1\{Y_i \leq y\} = 0$ if $y < y_1$. Part (ii) of the proposition follows by the same steps as Proposition 2.2 in \citet{BalabdaouiDurotJankowski2019}. Note that the minimizers $\hat{\alpha}_n$ and, hence, $\hat{F}_{n,\hat{\alpha}_n}$ do depend on $Q$, which appears in the criterion \eqref{eq:weighted_lsc}. To lighten the notation, we write $\hat{F}_{n,\hat{\alpha}_n} = \hat{F}_n$ in the following, and only use the subscript when it is necessary to indicate the dependence on $\hat{\alpha}_n$. To define $\hat{F}_{n}$ beyond the set $\{\hat{\alpha}_n^{\top}X_1, \dots, \hat{\alpha}_n^{\top}X_n\} \times \mathrm{supp}(Q)$, we let
\begin{equation} \label{eq:y_interpolation}
\hat{F}_{n}(z,y) = \begin{cases}
	0, & y < y_1, \\
	\hat{F}_{n}(z, y_j), & y \in [y_j, y_{j + 1}), \, j = 1, \dots, k - 1, \\
	1, & y \geq y_k,
\end{cases}
\end{equation}
for $z \in \{\hat{\alpha}_n^{\top}X_1, \dots, \hat{\alpha}_n^{\top}X_n\}$ and $y \in \mathbb{R}$, and
\[
\hat{F}_{n}(z,y) = \begin{dcases}
	\hat{F}_{n}(z_1,y), & z < z_1, \\
	\frac{z_{j+1}-z}{z_{j+1}-z_j}\hat{F}_{n}(z_j, y) + \frac{z-z_j}{z_{j+1}-z_j}\hat{F}_{n}(z_{j+1},y), & z \in [z_j, z_{j+1}), \, j = 1, \dots, m - 1, \\
	\hat{F}_{n}(z_m,y), & z \geq z_m.
\end{dcases}
\]
We apply these interpolation methods in our empirical studies in Section \ref{sec:empirical}. For the theory, any other interpolation methods satisfying the monotonicity constraints in both arguments is admissible.

In the forecasting literature, the loss function \eqref{eq:weighted_lsc} with $Q$ equal to the Lebesgue measure $\lambda$ is known under the name continuous ranked probability score (CRPS), which is a widely used proper scoring rule for the estimation of distribution functions and for forecast evaluation \citep{GneitingRaftery2007}. The criterion with general Borel measures $Q$ are the so-called threshold weighted forms of the CRPS \citep{GneitingRanjan2011}. At a first sight, the CRPS seems to be a natural choice for the loss function since it weighs all thresholds equally, but it has the drawback that $\mathbb{E}[L_n(\lambda; F_0, \alpha_0)]$ is finite only if the conditional distributions corresponding to $F_0(\alpha_0^{\top}x, \cdot)$ have finite first moment; see (21) in \citet{GneitingRaftery2007}. This is an unnecessary assumption if the goal is the estimation of the conditional CDFs, rather than conditional expectations, and it complicates proofs of consistency. We therefore focus on finite measures $Q$.

\section{Convergence rates} \label{sec:rates}

\subsection{Assumptions}
We proceed to establish consistency results for the bundled estimator $\hat{F}_n(\hat{\alpha}_n^{\top}x, y)$ and for the the separated estimators $\hat{\alpha}_n$ and $\hat{F}_n(z, y)$. The proofs and assumptions are closely related to those by \citet{BalabdaouiDurotJankowski2019} for the monotone single index model.

\begin{assumption} \label{assp:space}
	The set $\mathcal{X}$ is bounded and convex.
\end{assumption}

\begin{assumption} \label{assp:measure_Q}
The measure $Q$ and the distribution of $(X,Y)$ satisfy one of the following assumptions.
\begin{itemize}
    \item[(i)] The distribution of $X$ admits a Lebesgue density $p_X$ which is bounded from below by $\munderbar{p}_X > 0$ and from above by $\bar{p}_X < \infty$, and $Q$ has finite support, putting mass only on points $t_1 < \dots < t_p$.
    \item[(ii)] For all $y \in \mathrm{supp}(\mathbb{P}^Y)$, the distribution of $X$ conditional on $Y = y$ admits a Lebesgue density bounded from below by $\munderbar{p}_X > 0$ and from above by $\bar{p}_X < \infty$, with constants not depending on $y$. The measure $Q$ has support on $[a,b]$ and admits a Lebesgue density $q$ bounded from above by $c < \infty$.
\end{itemize}
\end{assumption}

\begin{assumption} \label{assp:differentiability}
	For all $t \in \mathrm{supp}(Q)$ the function $z \mapsto F_0(z,t)$ is continuously differentiable on $\mathcal{C}_{\alpha_0}$ with derivative $F_0^{(1)}(z, t)$, and $0 < |F_0^{(1)}(z,t)| \leq K_t$ for all $z \in \mathcal{C}^{\circ}_{\alpha_0}$ and some $K_t < \infty$.
\end{assumption}

\begin{assumption} \label{assp:distribution_alphax}
	For all $\alpha \in \mathcal{S}_{d-1}$, the random variable $\alpha^{\top}X$ admits a Lebesgue density bounded from below by $\munderbar{q} > 0$ and from above by $\bar{q} > 0$.
\end{assumption}

\begin{assumption} \label{assp:continuous_density}
    The density $p_X$ of $X$ is continuous on $\mathcal{X}$.
\end{assumption}

Assumptions \ref{assp:space}, \ref{assp:distribution_alphax} and \ref{assp:continuous_density} correspond to (A1), (A4) and (A6) in \citet{BalabdaouiDurotJankowski2019}, respectively, and Assumption \ref{assp:differentiability} is a direct extension of their condition (A5) to our case. \\

In the next sections,  we present one of the main convergence results of this work, derived under the assumptions above. The case $Q =\mathbb{P}^Y_n$ would have been a natural choice.  One referee raised the point of whether one can derive rates of convergence in this case when the distribution of $Y$ is  compactly supported.  Unfortunately, compactness of the support does not solve the issue that empirical process associated with the estimation problem at hand contains a term that cannot be handled with the classical results such as Lemma 3.4.2 or Lemma 3.4.3 of \citet{vanderVaartWellner1996}. The reason behind the additional difficulties is that this term in question is of the form
\begin{eqnarray*}
\int \psi_n(t)   (\mathbb{P}^Y_n - \mathbb{P}^Y)(t)
 \end{eqnarray*}
 where $ \psi_n$ is random function involving the empirical measure $\mathbb P_n$; additional details are provided in Section \ref{sec:transformation_invariant}.  More sophisticated tools need to be used in this context.  The problem is beyond the scope of this article but worth investigating in future research.

\subsection{Convergence rate for the bundled estimator}
The results for convergence rates for both types of $Q$ in Assumption \ref{assp:measure_Q} are presented in a unified framework. For the bundled estimator, we obtain the following result.

\begin{theorem} \label{thm:bundled}
Under Assumptions \ref{assp:space} and \ref{assp:measure_Q}, it holds that
\[
    \left(\int_{\mathbb{R}}\int_{\mathbb{R}}(\hat{F}_n(\hat{\alpha}_n^{\top}x, t) - F_0(\alpha_0^{\top}x, t))^2 \, d\mathbb{P}^X(x)dQ(t) \right)^{1/2} = \ O_p(n^{-1/3}).
\]
\end{theorem}

The proof of Theorem \ref{thm:bundled} applies Theorem 3.4.1 and Lemma 3.4.2 of \citet{vanderVaartWellner1996}, and it is given in Section \ref{sec:proof_bundled}. In the following, we introduce empirical process notation, provide auxiliary results that are of independent interest, and discuss the techniques and problems involved in the proof.

In accordance with Assumption \ref{assp:space}, assume $\|x\| \leq R$ for all $x \in \mathcal{X}$ and some $R > 0$, so that $|\alpha^{\top}x| \leq R$ for $\alpha \in \mathcal{S}^{d-1}$. In the proofs, the following function classes appear,
\begin{align*}
	\mathcal{H} & = \{h\colon [-2R, 2R] \rightarrow [0,1], \, \text{non-increasing}\}, \\
\mathcal{G} & = \{g\colon \mathcal{X} \rightarrow [0,1], \, g(x) = h(\alpha^{\top}x), \, (\alpha, h) \in \mathcal{S}_{d-1} \times \mathcal{H}\},
\end{align*}
where the support in the class $\mathcal{H}$ has to be extended to $[-2R, 2R]$ for technical reasons. Non-increasing functions $\tilde{h}\colon [-R, R] \rightarrow [0,1]$ are considered as elements of $\mathcal{H}$ by constant extrapolation at the boundaries. Denote the $L_2$-norm of functions from $\mathcal{X}$ to $\mathbb{R}$, with respect to a Borel measure $\mu$, by
\[
	\|f\|_{\mu} = \left(\int_{\mathcal{X}}f(x)^2\,\,d\mu(x)\right)^{1/2}.
\]
For integration with respect to the Lebesgue measure over a set $A$, we write $\|f\|_{A}$. The bracketing entropy of a function class $\mathcal{T}$ with respect to some  norm $\|\cdot\|$ is denoted by $N_B(\varepsilon, \mathcal{T}, \|\cdot\|)$, and the bracketing integral is defined as
\[
    \tilde{J}(\delta, \mathcal{T}, \|\cdot\|) = \int_0^{\delta}\sqrt{1+\log N_B(\varepsilon, \mathcal{T}, \|\cdot\|) \, d\varepsilon}.
\]
The following proposition, which relies on Theorem 2.7.5 of \citet{vanderVaartWellner1996} and a result of \citet{FeigeSchechtman2002}, is crucial for all our results. 

\begin{proposition} \label{prop:entropy}
Let $\mu$ be a Lebesgue continuous distribution with support in a bounded set contained in a ball of radius $R > 0$ with density bounded from above by $D > 0$. Then,
\[
	\log(N_B(\varepsilon, \mathcal{G}, \|\cdot\|_{\mathcal{X}})) \leq \frac{2^{(d+1)/2}d^{1/4}R^{(d-1)/2}(1+\sqrt{R})(d\sqrt{A}+K\sqrt{R})\sqrt{D}}{\varepsilon}
\]
for universal constants $A, K > 0$.
\end{proposition}

Due to Proposition \ref{prop:entropy}, the entropy of the class of functions $x \mapsto F(\alpha^{\top}x, y)$ for $(F,\alpha) \in \mathcal{F}$ and $y$ fixed is of the same order as the entropy of the monotone function class with values in $[0,1]$. If $Q$ has finite support, this is sufficient to obtain the cubic convergence rate. However, as one would expect, the constants in the bounds increase with the size of the support, and it is not possible to extend the same proof strategy to Lebesgue continuous $Q$. For this case, a bound for the entropy of the class
\begin{equation} \label{eq:difficult_entropy_clas}
    \mathcal{M} := \left\{h\colon \mathbb{R}^d\times \mathbb{R} \rightarrow [0,1], \, h(x,y) = \int_{[y,\infty)}F(\alpha^{\top}x, t)^2 dQ(t), \, (F,\alpha) \in \mathcal{F} \right\}
\end{equation}
is required. We find such a bound by constructing a suitable discretization of the support of $Q$.

\begin{remark}
One might think that a simpler way to bound the entropy of the class $\mathcal{M}$ would be via the results of \citet{GaoWellner2007} on the entropy of multivariate monotone function. Indeed, the function $(z,y) \mapsto F(z,y)$ is bivariate monotone, and due to Proposition \ref{prop:entropy}, the fact that we have $\alpha^{\top}x$ in the first argument only increases the entropy by a constant factor. However, according to Theorem 1.1 of \citet{GaoWellner2007}, the entropy of the class of bivariate monotone functions is of order $1/\varepsilon^2$, which leads to a diverging entropy integral. Even with the relaxation discussed on p.~326 of \citet{vanderVaartWellner1996}, which allows to integrate only from $\min(u\delta^2, \delta)/3$ for small $u > 0$ in the entropy integral, it is not possible to achieve the cubic rate with this entropy bound.
\end{remark}

\subsection{Convergence rate for the separated estimators}
The rate for the separated estimators $\hat{F}_n(z,y)$ and $\hat{\alpha}_n$ relies on Theorem \ref{thm:bundled} and is proved in a similar way as in Theorem 5.2 and Corollary 5.3 of \citet{BalabdaouiDurotJankowski2019}. Note that under our model assumptions, the parameters $F$ and $\alpha$ are indeed identifiable. More precisely, if $F(\alpha^{\top}X,t) = F_0(\alpha_0^{\top}X,t)$ almost surely for a fixed $t$, then $F(z,t) = F_0(z,t)$  for $(z,t) \in \mathcal{C}_{\alpha_0} \times \mathrm{supp}(Q)$, and $\alpha = \alpha_0$. This is shown in an analogous way as in Proposition 5.1 of \citet{BalabdaouiDurotJankowski2019}, and it is proven in Section \ref{sec:identifiability} for completeness.

\begin{theorem}
\label{thm:separated_consistency}
    Let Assumptions \ref{assp:space}, \ref{assp:measure_Q} and \ref{assp:distribution_alphax} hold true. Assume that for each $t$ the function $F_0(\cdot, t)$ is left-continuous, non constant and does not have discontinuity points on the boundary of $\mathcal{C}_{\alpha_0}$. Furthermore, assume that from each subsequence $(n_k)_{k\in\mathbb{N}}$ we can extract another subsequence $(n_{k_l})_{l\in\mathbb{N}}$ which satisfies
    \begin{equation}
    \label{eq:asconsistencyhard}
	\lim_{l \to \infty}  \int_\mathbb{R} \int_\mathcal{X}(F_0(\alpha_0^{\top}x,t) - \hat{F}_{n_{k_l}}(\hat{\alpha}_{k_l}^{\top}x,t) )^2d\mathbb{P}^X(x) dQ(t) = 0
    \end{equation}
    almost surely. Then,
    \begin{itemize}
	\item[(i)] $\hat{\alpha}_n$ converges to $\alpha_0$ in              probability in the euclidean norm,
	\item[(ii)] for all continuity points $(z,t)$ of $F_0$ in $\mathcal{C}^{\circ}_{\alpha_0} \times \mathrm{supp}(Q)$, we have that           $\hat{F}_{n}(z,t)$ converges to              $F_0(z,t)$ in probability.
    \end{itemize}
\end{theorem}
\begin{remark}
    The condition \eqref{eq:asconsistencyhard} in Theorem \ref{thm:separated_consistency} holds under our assumptions due to Theorem \ref{thm:bundled}.
\end{remark}

\begin{theorem}
\label{thm:separated_rate}
    Define $\underline{c} = \inf \mathcal{C}_{\alpha_0}$ and $\overline{c} = \sup \mathcal{C}_{\alpha_0}$. Under Assumptions \ref{assp:space}-\ref{assp:continuous_density}, we have that
\begin{itemize}
	\item[(i)] $\lVert \alpha_0 - \hat{\alpha}_n \lVert = O_P(n^{-1/3})$;
	\item[(ii)] if $\sup_{t\in\mathrm{supp}(Q)} K_t < \infty$, then
		\begin{equation}
    \label{eq:separatesampleslitting}
			\left(\int_{\R} \int_{\underline{c} + v_n}^{\overline{c} - v_n} \left(F_0(z,t) - \hat{F}_{n}(z,t)\right)^2 dz\ dQ(t) \right)^{1/2} = O_P(n^{-1/3})
		\end{equation}
		for all sequences $v_n$ such that $\underline{c} + v_n\leq\overline{c} - v_n$ and $n^{1/3} v_n \to \infty$ for $n \to \infty$. 
\end{itemize}
\end{theorem}

Part (ii) of Theorem \ref{thm:separated_rate} can be regarded as analogous to the result \eqref{eq:dim_old_rate} derived by \citet[Theorem 5.1]{HenziKlegerZiegel2023}, with the weighted $L_2$-norm replacing the supremum norm. \citet{HenziKlegerZiegel2023} do not assume a linear index function, but they impose the assumption that the index function is estimated the rate of $(\log(n)/n)^{1/2}$, rather than deriving a convergence rate, which we do in part (i) of the above theorem.

\section{Empirical distribution as weighting measure} \label{sec:transformation_invariant}
The methods proposed so far require the specification of a weighting measure $Q$. An interesting variant of the criterion \eqref{eq:weighted_lsc}, which does not require an explicit weighting choice, arises when $Q$ equals the empirical distribution $\mathbb{P}_n^Y$; that is,
\begin{equation*}
	L_n(\mathbb{P}_n^Y; F, \alpha) = \frac{1}{n} \sum_{i=1}^n \int_{\mathbb{R}} (1\{Y_i \leq t\} - F(\alpha^{\top}X_i, t))^2 \, d\mathbb{P}_n^Y(t) = \frac{1}{n^2} \sum_{i,j=1}^n (1\{Y_i \leq Y_j\} - F(\alpha^{\top}X_i, Y_j))^2.
\end{equation*}
According to the follwing lemma, for this choice of $Q$ the estimator $\hat{\alpha}_n$ and the pointwise error of the CDFs at the observed values of the response variable do not depend on the scale of the observations $Y$.
\begin{lemma} \label{lem:Invariance}
Let $f\colon \mathbb{R} \rightarrow \mathbb{R}$ be strictly increasing on the support of $Y$, and $f^{-1}(t) = \inf\{s \in \mathbb{R}\colon f(s) \geq t\}$. Then, the following hold with probability one.
\begin{itemize}
	\item[(i)] A tuple $(\hat{F}_{n, \hat{\alpha}_n}, \hat{\alpha}_n)$ minimizes $L_n(\mathbb{P}_n^Y; \cdot)$ if and only if $(\tilde{F}_{n, \hat{\alpha}_n}, \hat{\alpha}_n)$ with $\tilde{F}_{n, \hat{\alpha}_n}(t, z) = \hat{F}_{n, \hat{\alpha}_n}(z, f^{-1}(t))$ is a minimizer of $L_n(\mathbb{P}_n^{f(Y)}; \cdot)$, and it holds that $L_n(\mathbb{P}_n^{Y}; \hat{F}_{n, \hat{\alpha}_n}, \hat{\alpha}_n) = L_n(\mathbb{P}_n^{f(Y)}; \tilde{F}_{n, \hat{\alpha}_n}, \hat{\alpha}_n)$.
	\item[(ii)] With $\tilde{F}_0(z,t) = F_0(z, f^{-1}(t))$ and $t_i = f(Y_i)$, we have
	\[
		\tilde{F}_{n,\hat{\alpha}_n}(\hat{\alpha}_n^{\top}X_i, t_i) - \tilde{F}_0(\alpha^{\top}X_i, t_i) = \hat{F}_{n,\hat{\alpha}_n}(\hat{\alpha}_n^{\top}X_i, Y_i) - F_0(\alpha^{\top}X_i, Y_i), \ i = 1, \dots, n.
	\]
	If $\tilde{F}_{n,\hat{\alpha}_n}$ and $\hat{F}_{n,\hat{\alpha}_n}$ are interpolated as in \eqref{eq:y_interpolation}, then the above equality holds for all $y \in \mathbb{R}$ and $t = f(y)$.
\end{itemize}
\end{lemma}
The above result is generally not true for $Q \neq \mathbb{P}_n^Y$ in \eqref{eq:weighted_lsc}. The invariance property aligns well with the fact that the transformed outcome $f(Y)$ again follows a distributional single index model with the same parameter $\alpha_0$ and the corresponding CDFs $t \mapsto F_0(\alpha_0^{\top}x, f^{-1}(t))$. However, it turns out that deriving convergence rates for this criterion is substantially more difficult than for fixed measures $Q$, because the integral in the function class $\mathcal{M}$ in \eqref{eq:difficult_entropy_clas} is now over the random measure $\mathbb{P}_n$ instead of the fixed measure $Q$. We suspect that the rate for this estimator should still be of order $O_p(n^{-1/3})$, and our simulations confirm this intuition in certain examples. However, a completely different strategy of proof seems necessary to prove this rate.

\section{Empirical results} \label{sec:empirical}

\subsection{Simulations} \label{sec:simulations}

\begin{table}
	\centering
 \footnotesize
	\begin{tabular}{c|l|l|ccc|ccc}
		\toprule
		$Q$ & Simulation & Error type & \multicolumn{6}{c}{Spherical coordinates $\theta_0$} \\[0.2em]
	& & & $\pi/4$ & $\pi/3$ & $\pi/2$ & $(\pi/4, \pi/2)$ & $(\pi/3,\pi/3)$ & $(\pi/2, \pi/4)$ \\
 \midrule
\multirow[c]{6}{*}{Empirical} &  & Index & 0.49 (0.04) & 0.50 (0.04) & 0.63 (0.05) & 0.48 (0.03) & 0.47 (0.03) & 0.49 (0.03) \\
& Exponential & CDF & 0.25 (0.01) & 0.23 (0.01) & 0.36 (0.01) & 0.27 (0.01) & 0.19 (0.02) & 0.27 (0.01) \\
&  & Bundled & 0.36 (0.01) & 0.36 (0.01) & 0.37 (0.01) & 0.37 (0.01) & 0.37 (0.01) & 0.37 (0.01) \\[0.25em]
&  & Index & 0.49 (0.04) & 0.50 (0.05) & 0.60 (0.05) & 0.48 (0.03) & 0.55 (0.03) & 0.48 (0.03) \\
& Gaussian & CDF & 0.30 (0.01) & 0.30 (0.01) & 0.37 (0.02) & 0.33 (0.01) & 0.33 (0.01) & 0.34 (0.01) \\
&  & Bundled & 0.38 (0.01) & 0.38 (0.01) & 0.40 (0.01) & 0.39 (0.01) & 0.39 (0.01) & 0.39 (0.01) \\
\midrule
\multirow[c]{6}{*}{Truncated} &  & Index & 0.50 (0.04) & 0.53 (0.04) & 0.54 (0.05) & 0.51 (0.03) & 0.48 (0.03) & 0.46 (0.03) \\
& Exponential & CDF & 0.24 (0.02) & 0.24 (0.02) & 0.38 (0.03) & 0.29 (0.02) & 0.21 (0.02) & 0.27 (0.02) \\
&  & Bundled & 0.37 (0.01) & 0.37 (0.01) & 0.40 (0.02) & 0.39 (0.01) & 0.39 (0.01) & 0.40 (0.01) \\[0.25em]
&  & Index & 0.46 (0.04) & 0.48 (0.05) & 0.55 (0.05) & 0.48 (0.03) & 0.55 (0.03) & 0.50 (0.03) \\
& Gaussian & CDF & 0.28 (0.01) & 0.29 (0.02) & 0.37 (0.02) & 0.32 (0.02) & 0.30 (0.01) & 0.33 (0.02) \\
&  & Bundled & 0.38 (0.01) & 0.38 (0.01) & 0.41 (0.01) & 0.39 (0.01) & 0.40 (0.01) & 0.40 (0.01) \\
\midrule
\multirow[c]{6}{*}{Uniform} &  & Index & 0.56 (0.04) & 0.55 (0.04) & 0.61 (0.05) & 0.45 (0.03) & 0.49 (0.03) & 0.48 (0.03) \\
& Exponential & CDF & 0.21 (0.02) & 0.21 (0.02) & 0.38 (0.02) & 0.27 (0.03) & 0.19 (0.03) & 0.24 (0.03) \\
&  & Bundled & 0.37 (0.01) & 0.36 (0.01) & 0.39 (0.01) & 0.37 (0.01) & 0.38 (0.01) & 0.39 (0.01) \\[0.25em]
&  & Index & 0.49 (0.04) & 0.46 (0.04) & 0.56 (0.05) & 0.48 (0.03) & 0.53 (0.03) & 0.49 (0.03) \\
& Gaussian & CDF & 0.26 (0.02) & 0.27 (0.02) & 0.38 (0.02) & 0.29 (0.02) & 0.22 (0.02) & 0.29 (0.02) \\
&  & Bundled & 0.38 (0.01) & 0.38 (0.01) & 0.41 (0.01) & 0.39 (0.01) & 0.39 (0.01) & 0.39 (0.01) \\
\bottomrule
	\end{tabular}
	\caption{Estimated convergence rates and standard errors (in parentheses) in the simulation examples \eqref{eq:sim_examples}, with $\theta_0$ giving the spherical coordinates for $\alpha_0$. The rates are the slope coefficient from regressing $\log(\mathrm{err}_{n,i})$ on $-\!\log(n)$, with $n = 2^8, 2^9, \dots, 2^{13}$ and samples $i = 1, \dots, 100$, for each error measure, simulation setting, and weighting measure for our estimator. The rate in the index estimation is mostly close to $1/2$, while the rates for the other measures are around $1/3$. \label{tab:estimated_rates}}
\end{table}

We investigate the convergence of our estimators in simulations. For $d = 2, 3$, we simulate $X_j \sim \mathrm{Unif}(0,1)$, $j = 1, \dots, d$, independently, and generate the response variable in two ways,
\begin{equation} \label{eq:sim_examples}
	Y^{(1)} = (\alpha_0^{\top}X)^3 \varepsilon, \ \varepsilon \sim \mathcal{N}(0,1), \qquad Y^{(2)} =  (\alpha_0^{\top}X)^3\eta, \ \eta \sim \mathrm{Exp}(1).
\end{equation}
For the weighting measure $Q$, we consider the empirical distribution $\mathbb{P}^Y_n$, the uniform distribution on $[-10, 10]$ and the Gaussian distribution with variance $4$ truncated to the interval $[-4, 10]$ for the simulations with Gaussian noise, and the uniform distribution on $[0,50]$ and the truncated Gamma distribution with shape $3$ and scale $1$ for the simulations with exponentially distributed noise, respectively. The rationale is that the uniform distribution over a large set provides a rather rough choice for the weighting, whereas the truncated distributions more closely follow the actual outcome distributions, up to truncation to a compact interval.

The index $\alpha_0$ is parameterized in spherical coordinates with $\theta_0 \in [0,2\pi]$ and values $\theta_0 = \pi/4, \, \pi/3, \, \pi/2$ for $d = 2$, and $\theta_0 \in [0,\pi] \times [0,2\pi]$ and values $\theta_0 = (\pi/4, \pi/2), \, (\pi/3,\pi/3), \, (\pi/2, \pi/4)$ for $d = 3$. To perform estimation, we parameterize $\alpha$ in spherical coordinates and do a grid search followed by local numerical optimization. For $d = 2$, we choose $40$ equidistant points $\theta_1 = 0 < \theta_2 < \dots < \theta_{40} = 2\pi$, evaluate the criterion \eqref{eq:weighted_lsc} at $\alpha_j = (\cos(\theta_j), \sin(\theta_j))$, and perform numerical optimization of $\eqref{eq:weighted_lsc}$ with respect to $\theta$ in $\alpha = (\cos(\theta), \sin(\theta))$ around the $\theta_j$ for which the minimal value of the criterion is attained. The procedure for $d = 3$ is analogous, and for the grid we take all combinations of $20$ equidistant points $\theta_{1,j} \in [0,\pi]$ and $40$ points $\theta_{2,k} \in [0,2\pi]$, $j = 1, \dots, 20$, $k = 1, \dots, 40$. Numerical optimization is performed with {\tt optimize} in \textsf{R} \citep{Rcore2022} for $d = 2$, and {\tt nmkb} from the package {\tt dfoptim} \citep{Varadhan2020} for $d = 3$. Estimation of the conditional CDFs uses the {\tt isodistrreg} package \citep{HenziZiegelGneiting2021}. A general implementation of our estimator and replication material for Section \ref{sec:empirical} are available on \url{https://github.com/AlexanderHenzi/distr_single_index}.

To estimate the rates of convergence, we simulate $100$ realizations of the examples described above for each of the sample sizes $n = 2^m$, $m = 8, 9, \dots, 13$, and compute the the index error $\|\hat{\alpha}_n - \alpha_0\|$, the bundled error $L(\hat{F_n}, \hat{\alpha}_n)$, and of the error of the CDFs $L_{\mathrm{CDF}}(\hat{F}_n)$. The integrals in $L(\hat{F_n}, \hat{\alpha}_n)$ and $L_{\mathrm{CDF}}(\hat{F}_n)$ are estimated with the mean of the integrand evaluated at $5000$ draws for $y \sim \mathbb{P}^Y$ and $x \sim \mathbb{P}^X$, or $z \sim \mathrm{Unif}(\munderbar{c}, \bar{c})$, respectively. We then estimate the convergence rate with the slope coefficient from regressing $\log(\mathrm{err}_{n,i})$, for all $n$ and samples $i = 1, \dots, 100$, on $-\!\log(n)$, for each setting and error measure. The estimates and standard errors are shown in Table \ref{tab:estimated_rates}. Naturally, there are many factors influencing the convergence rates estimates, such as noise in the estimation, different constants in different examples, and, most importantly, the fact that the rates of the errors are only estimated on a grid of finite sample sizes. Therefore, even if one might expect the same asymptotic rates in the examples that we consider, there are some deviations due to different constants and finite sample effects. However, Table \ref{tab:estimated_rates} suggests that the rate of $\hat{\alpha}_n$ is faster than $n^{-1/3}$, as in the experiments of \citet{BalabdaouiDurotJankowski2019}, and the rates for the bundled estimator and for the CDF are around $n^{-1/3}$. There are no systematic differences between the results for $Q = \mathbb{P}_n^Y$ and for the other approaches, with average rates over all settings of $0.51$, $0.28$, and $0.39$ for the index, CDF, and bundled estimator for the empirical weighting measure, and $0.51$, $0.30$, and $0.38$ for the other weighting methods. This suggests that the same rates should hold for $Q = \mathbb{P}^Y_n$.

\begin{remark}
For dimension $d = 1$, the computation of $\hat{\alpha}_n$ is a one-dimensional optimization problem, and $\hat{\alpha}_n$ can be approximated to a high accuracy provided that the grid for the initial grid search is fine enough. For $d > 2$ the grid search becomes expensive, and there are no guarantees that a pair $(\tilde{\alpha}_n, \tilde{F}_n)$ chosen by our implementation is a global minimizer of our target function, which is non-smooth and non-convex. Estimation in the monotone single index model for the mean suffers from the same optimization difficulties, and although there has been extensive research on implementation and alternative methods for estimating $\hat{\alpha}_n$ \citep{Groeneboom2018, BalabdaouiGroeneboomHendrickx2019, GroeneboomHendrickx2019, BalabdaouiGroeneboom2021}, the computation of $\hat{\alpha}_n$ remains a challenge, especially in higher dimensions.
\end{remark}

\subsection{Illustration on house price data} \label{sec:data}

\begin{figure}[t]
	\includegraphics[width=\textwidth]{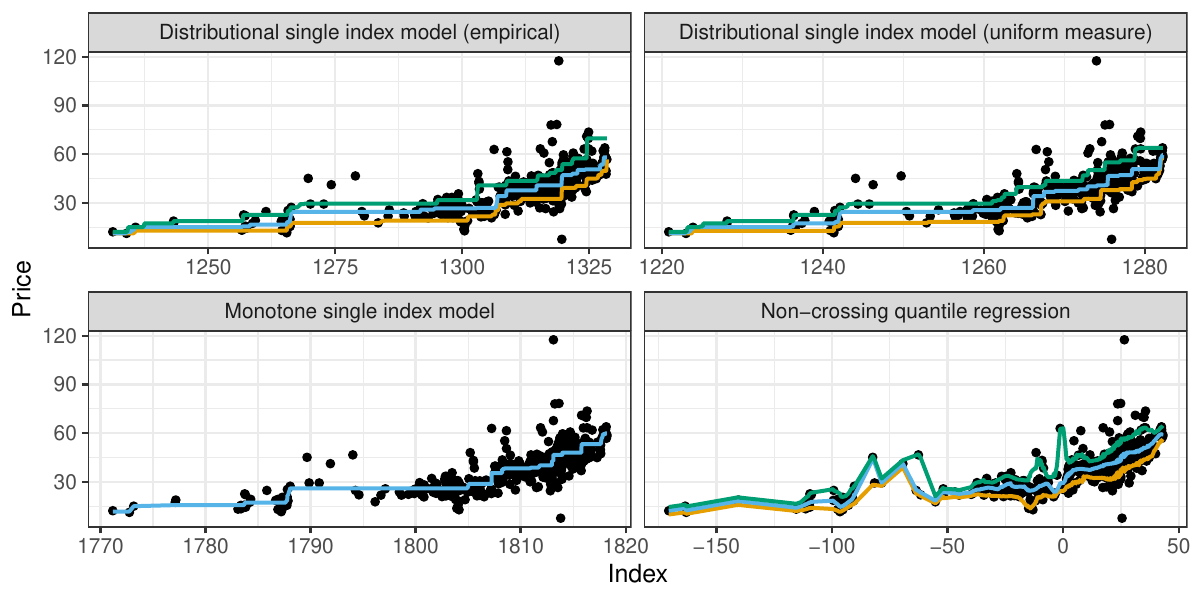}
	\caption{Pairs $(\hat{\alpha}_n^{\top}X_i, Y_i)$, $i = 1, \dots, 414$, for the distributional single index model, the monotone single index model, and for non-crossing quantile regression. The lines for the distributional methods are estimated conditional quantile curves at the levels $\tau = 0.1, 0.5, 0.9$. \label{fig:quantiles_plot}}
\end{figure}

We illustrate the distributional single index model in a data example by \citet[Section 4.4]{Jiang2023}. The data set, which is available on \url{https://doi.org/10.24432/C5J30W}, contains 414 real estate transaction records from Tapiei City and New Taipei City. The dependent variable is the price per unit area, and the covariates are the number of convenience stores in the living circle on foot, the building age, the transaction year and month, and the distance to the nearest metro station. The transaction time is transformed to a numerical variable with values in between $2012.67$ and $2013.58$, and it is a proxy variable which captures effects such as trends in the house prices, or different policy regimes over time that might influence the prices.

Figure \ref{fig:quantiles_plot} depicts the index values $\hat{\alpha}_n^{\top}X_i$ and prices $Y_i$, $i = 1, \dots, 414$, for the distributional single index model, the monotone single index model, and for the non-crossing quantile regression estimator by \citet{Jiang2023}; the results for the latter are equal to their Figure 3 (c) and reproduced with the code from the supplement of their article. We implemented the distributional index model with the empirical measure and with the uniform measure over a large set including all observed prices. For the distributional methods, the lines in the figure show estimated conditional quantiles at levels $\tau = 0.1, 0.5, 0.9$, which are obtained by inversion of the CDFs for our estimator. \citet{Jiang2023} center all covariates around their mean before estimation. With shape restricted estimation methods, such centering is not necessary since it does not change the order of the projections $\hat{\alpha}_n^{\top}X_i$. As the scatterplots suggests, the order of the index values $\hat{\alpha}_n^{\top}X_i$, $i = 1, \dots, 414$, obtained with the three methods are very similar, and the pairwise Spearman correlations between them are indeed all above $0.98$.
In the given data application, all methods have advantages and disadvantages. The computation of the estimator by \citet{Jiang2023} is fast, but it involves several tuning parameters, namely, an initial quantile level for estimation, set to $\tau = 0.5$, bandwidths for kernel smoothing, and a pre-specified grid of quantiles on which the estimator is computed and evaluated, chosen to be $\tau = 0.1, 0.2, \dots, 0.9$. Estimation for our method and for the monotone single index model is slower, since we take a fine grid for the grid search over $\alpha$ and perform local optimization in several regions to ensure a good approximation of the minimum. However, the parameters of the shape restricted methods are more easily interpretable due to the monotone dependence on $\hat{\alpha}_n^{\top}X$. One can draw the --- reasonable --- conclusions that the price is increasing in the number of closely situated convenience stores and over time, and decreasing in the distance to the nearest metro station and in the age of the building; see Table \ref{tab:index_estimates}. The interpretation is more difficult for the estimator by \citet{Jiang2023}. Although the signs of $\hat{\alpha}_n$ in their estimator agree with those of the shape restricted methods, the conditional quantile curves are non-monotone and interpolate the prices for some of the observations.

\begin{table}
	\centering
	\small
	\begin{tabular}{l| cccc}
		\toprule
		Method & Number stores & Building age & Transaction date & Distance metro\\
		\midrule
        Distributional index model (empirical) & 0.706 & -0.263 & 0.658 & -0.013\\
        Distributional index model (uniform) & 0.750 & -0.186 & 0.634 & -0.008\\
        Monotone single index model & 0.415 & -0.122 & 0.902 & -0.006\\
        Non-crossing quantile regression & 0.152 & -0.060 & 0.987 & -0.004\\
		\bottomrule
	\end{tabular}
	\caption{Estimates $\hat{\alpha}_n$ for the different methods. For the non-crossing quantile regression, we show the entries for $\mathrm{NCCQR}_9$ from Table 6 of \citet{Jiang2023}, standardized to norm $1$ for comparability. \label{tab:index_estimates}}
\end{table}

\section{Discussion} \label{sec:discussion}
In this article, we proposed estimators for the distributional single index model, and proved a convergence rate of $O_P(n^{-1/3})$ both for bundled and separated estimators. This greatly improves upon the $(\log(n)/n)^{1/6}$-rate known so far. There are several avenues for future research. Consistency for our transformation-invariant estimator proposed in Section \ref{sec:transformation_invariant} is an open challenge, which goes beyond the techniques applied for the convergence rates in this article. A possible future research direction is to study convergence under more general weighting measures $Q$ with possibly an unbounded support. This would allow analyzing whether there is an optimal choice of $Q$ in terms of the estimation error for $\alpha_0$. As for the monotone single index model, our simulations also suggest that $\alpha_0$ is estimated at a faster rate. Deriving this rate, as well as a comparison to the estimators for $\alpha_0$ in the monotone single index model, would be an interesting direction for future work.

\section*{Acknowledgments}
We are grateful to an anonymous referee for helpful comments.

\setlength{\bibsep}{0pt plus 0.3ex}
\bibliographystyle{apalike}
\bibliography{references}

\begin{thebibliography}{}

\bibitem[Balabdaoui et~al., 2019a]{BalabdaouiDurotJankowski2019}
Balabdaoui, F., Durot, C., and Jankowski, H. (2019a).
\newblock Least squares estimation in the monotone single index model.
\newblock {\em Bernoulli}, 25(4B):3276--3310.

\bibitem[Balabdaoui and Groeneboom, 2021]{BalabdaouiGroeneboom2021}
Balabdaoui, F. and Groeneboom, P. (2021).
\newblock Profile least squares estimators in the monotone single index model.
\newblock In {\em Advances in contemporary statistics and econometrics ---
  {F}estschrift in honor of {C}hristine {T}homas-{A}gnan}, pages 3--22.
  Springer, Cham.

\bibitem[Balabdaoui et~al., 2019b]{BalabdaouiGroeneboomHendrickx2019}
Balabdaoui, F., Groeneboom, P., and Hendrickx, K. (2019b).
\newblock Score estimation in the monotone single-index model.
\newblock {\em Scand. J. Stat.}, 46(2):517--544.

\bibitem[Feige and Schechtman, 2002]{FeigeSchechtman2002}
Feige, U. and Schechtman, G. (2002).
\newblock On the optimality of the random hyperplane rounding technique for max
  cut.
\newblock {\em Random Structures \& Algorithms}, 20(3):403--440.

\bibitem[Gao and Wellner, 2007]{GaoWellner2007}
Gao, F. and Wellner, J.~A. (2007).
\newblock Entropy estimate for high-dimensional monotonic functions.
\newblock {\em J. Multivariate Anal.}, 98(9):1751--1764.

\bibitem[Gneiting and Raftery, 2007]{GneitingRaftery2007}
Gneiting, T. and Raftery, A.~E. (2007).
\newblock Strictly proper scoring rules, prediction, and estimation.
\newblock {\em J. Amer. Statist. Assoc.}, 102(477):359--378.

\bibitem[Gneiting and Ranjan, 2011]{GneitingRanjan2011}
Gneiting, T. and Ranjan, R. (2011).
\newblock Comparing density forecasts using threshold- and quantile-weighted
  scoring rules.
\newblock {\em J. Bus. Econom. Statist.}, 29(3):411--422.

\bibitem[Groeneboom, 2018]{Groeneboom2018}
Groeneboom, P. (2018).
\newblock Algorithms for computing estimates in the single index model.
\newblock \url{https://github.com/pietg/single_index}.

\bibitem[Groeneboom and Hendrickx, 2019]{GroeneboomHendrickx2019}
Groeneboom, P. and Hendrickx, K. (2019).
\newblock Estimation in monotone single-index models.
\newblock {\em Stat. Neerl.}, 73(1):78--99.

\bibitem[H\"{a}rdle et~al., 1993]{HardleHallIchimura1993}
H\"{a}rdle, W., Hall, P., and Ichimura, H. (1993).
\newblock Optimal smoothing in single-index models.
\newblock {\em Ann. Statist.}, 21(1):157--178.

\bibitem[Henzi et~al., 2023]{HenziKlegerZiegel2023}
Henzi, A., Kleger, G.-R., and Ziegel, J.~F. (2023).
\newblock Distributional (single) index models.
\newblock {\em J. Amer. Statist. Assoc.}, 118(541):489--503.

\bibitem[Henzi et~al., 2021]{HenziZiegelGneiting2021}
Henzi, A., Ziegel, J.~F., and Gneiting, T. (2021).
\newblock Isotonic distributional regression.
\newblock {\em J. R. Stat. Soc. Ser. B. Stat. Methodol.}, 83(5):963--993.

\bibitem[Jiang and Yu, 2023]{Jiang2023}
Jiang, R. and Yu, K. (2023).
\newblock No-crossing single-index quantile regression curve estimation.
\newblock {\em J. Bus. Econom. Statist.}, 41(2):309--320.

\bibitem[Lavri\v{c}, 1993]{Lavric1993}
Lavri\v{c}, B. (1993).
\newblock Continuity of monotone functions.
\newblock {\em Arch. Math. (Brno)}, 29(1-2):1--4.

\bibitem[M\"{o}sching and D\"{u}mbgen, 2020]{MoeschingDuembgen2020}
M\"{o}sching, A. and D\"{u}mbgen, L. (2020).
\newblock Monotone least squares and isotonic quantiles.
\newblock {\em Electron. J. Stat.}, 14(1):24--49.

\bibitem[Murphy et~al., 1999]{MurphyVanderVaartWellner1999}
Murphy, S.~A., van~der Vaart, A.~W., and Wellner, J.~A. (1999).
\newblock Current status regression.
\newblock {\em Math. Methods Statist.}, 8(3):407--425.

\bibitem[{R Core Team}, 2022]{Rcore2022}
{R Core Team} (2022).
\newblock {\em R: A Language and Environment for Statistical Computing}.
\newblock R Foundation for Statistical Computing, Vienna, Austria.

\bibitem[van~der Vaart, 1998]{vanderVaart1998}
van~der Vaart, A.~W. (1998).
\newblock {\em Asymptotic statistics}, volume~3 of {\em Cambridge Series in
  Statistical and Probabilistic Mathematics}.
\newblock Cambridge University Press, Cambridge.

\bibitem[van~der Vaart and Wellner, 1996]{vanderVaartWellner1996}
van~der Vaart, A.~W. and Wellner, J.~A. (1996).
\newblock {\em Weak convergence and empirical processes}.
\newblock Springer Series in Statistics. Springer-Verlag, New York.
\newblock With applications to statistics.

\bibitem[Varadhan et~al., 2020]{Varadhan2020}
Varadhan, R., Borchers, H.~W., and Bechard, V. (2020).
\newblock {\em dfoptim: Derivative-Free Optimization}.
\newblock R package version 2020.10-1.

\end{thebibliography}

\appendix

\section{Proofs} \label{sec:proofs}

\subsection{Proof of Theorem \ref{thm:bundled}} \label{sec:proof_bundled}

The proof of Theorem \ref{thm:bundled} is slightly different for the two cases in Assumption \ref{assp:measure_Q}, which involve different entropy calculations. We first give a proof for the theorem with an unspecified constant in an entropy bound, and then derive the constant for the two cases in separate lemmas.

\begin{proof}[Proof of Theorem \ref{thm:bundled}]
The proof applies Theorem 3.4.1 and Lemma 3.4.2 of \citet{vanderVaartWellner1996}.
\begin{align*}
    \mathbb{M}_n(F, \alpha) & = \int_{\mathbb{R}}\int_{\mathbb{R}^d \times \mathbb{R}}(1\{y \leq t\} - F(\alpha^{\top}x, t))^2 \, d\mathbb{P}_n(x,y)dQ(t), \\
    \mathbb{M}(F, \alpha) & = \int_{\mathbb{R}}\int_{\mathbb{R}^d \times \mathbb{R}}(1\{y \leq t\} - F(\alpha^{\top}x, t))^2 \, d\mathbb{P}(x,y)dQ(t).
\end{align*}
Expanding the squares and using the fact that $\mathbb{E}[1\{Y \leq t\} | X = x] = F_0(\alpha_0^{\top}x, t)$ yields
\[
    \mathbb{M}(F,\alpha) - \mathbb{M}(F_0, \alpha_0) = \int_{\mathbb{R}}\int_{\mathbb{R}}(F(\alpha^{\top}x, t) - F_0(\alpha_0^{\top}x, t))^2 \, d\mathbb{P}(x)dQ(t) =: d((F,\alpha), (F_0, \alpha_0))^2.
\]
Furthermore, we have
\[
    \mathbb{M}_n(F, \alpha) - \mathbb{M}(F, \alpha) = \int_{\mathbb{R}}\int_{\mathbb{R}^d \times \mathbb{R}}(1\{y \leq t\} - F(\alpha^{\top}x, t))^2 dQ(t) \, d(\mathbb{P}_n(x,y)-\mathbb{P}(x,y)),
\]
or, when rescaling with $\sqrt{n}$ and using empirical process notation,
\[
    \sqrt{n}\left(\mathbb{M}_n(F, \alpha) - \mathbb{M}(F, \alpha)\right) = \mathbb{G}_n \int_{\mathbb{R}}(1\{y \leq t\} - F(\alpha^{\top}x, t))^2 dQ(t)
\]
We now analyze the functions of the form
\[
    \ell(x,y) = \int_{\mathbb{R}}(1\{y \leq t\} - F(\alpha^{\top}x, t))^2 dQ(t)
\]
with $(F, \alpha) \in \mathcal{F}$, and denote the class of such functions by $\mathcal{L}$. Also, let $\mathcal{L}_{\delta}$ contain all functions of type
\[
    \tilde{\ell}(x,y) = \ell(x,y) - \int_{\mathbb{R}}(1\{y \leq t\} - F_0(\alpha_0^{\top}x, t))^2 dQ(t)
\]
with $\ell \in \mathcal{L}$ and for which
\[
    \delta^2 \geq \|\tilde{\ell}\|_{\mathbb{P}}^2 = d((F,\alpha), (F_0, \alpha_0))^2 = \mathbb{M}(F,\alpha) - \mathbb{M}(F_0, \alpha_0).
\]
The elements in $\mathcal{L}_{\delta}$ are obtained by shifting elements of $\mathcal{L}$ by a fixed function, so we have $N_B(\varepsilon, \mathcal{L}_{\delta}, \|\cdot\|_{\mathbb{P}}) \leq N_B(\varepsilon, \mathcal{L}, \|\cdot\|_{\mathbb{P}})$. To apply Lemma 3.4.2 of \citet{vanderVaartWellner1996}, we have to find an upper bound for the bracketing entropy of the class $\mathcal{L}$.
Since $Q$ is a finite measure, we have
\begin{align*}
    \ell(x,y) & = \underbrace{1-Q([t,\infty))}_{=:f(t)}
    + \underbrace{\int_{\mathbb{R}}F(\alpha^{\top}x, t)^2 dQ(t)}_{=:g(x)}
    + \underbrace{\int_{[y,\infty)}F(\alpha^{\top}x, t) dQ(t)}_{=:h(x,y)}.
\end{align*}
The function $f$ above does not contribute to the entropy, and $g$ does not depend on $y$ and belongs to the class $\mathcal{G}$, for which we know from Assumptions \ref{assp:space} and \ref{assp:measure_Q} and Proposition \ref{prop:entropy} that $\log(N_B(\varepsilon, \mathcal{G}, \|\cdot\|_{\mathbb{P}^X})) \leq \tilde{C}/\varepsilon$ for a constant $\tilde{C} > 0$. In separate lemmas below, we show that the entropy of the functions of the form $h$ above, with $(F,\alpha) \in \mathcal{F}$, is bounded from above by $\tilde{D}/\varepsilon$ for some constant $\tilde{D}$. 
Let now $[l, u]$ be an $\varepsilon$-bracket containing $g$ and $[L, U]$ an $\varepsilon$-bracket containing $h$. We interpret $l,u$ as functions of $(x,y)$ which are constant in $y$. Then the functions $U + u + 1-Q([t,\infty)$, $L + l + 1-Q([t,\infty)$ form a $(2\varepsilon)$-bracket containing $\ell$, because
\begin{align*}
    \|U + u - L - l\|_{\mathbb{P}}^2 & =  \int_{\mathbb{R}^d \times \mathbb{R}} \left \{ (U-L)^2 + (u-l)^2 + 2(U-L)(u-l) \right \} \, d \mathbb{P}(x,y) \\
    & \leq 2\varepsilon^2 + 2 \left(\int_{\mathbb{R}^d \times \mathbb{R}}(U-L)^2\, d \mathbb{P}(x,y)\right)^{1/2}\left(\int_{\mathbb{R}^d \times \mathbb{R}}(u-l)^2 \, d \mathbb{P}(x,y)\right)^{1/2} \\
    & \leq 4\varepsilon^2.
\end{align*}
Consequently, the number of $\varepsilon$-brackets required to cover $\mathcal{L}$ is bounded from above by $2(\tilde{C} +\tilde{D})/\varepsilon =: \kappa/\varepsilon$, which yields the following bound on the entropy integral,
\[
    \tilde{J}(\delta, \mathcal{L}, \|\cdot\|_{\mathbb{P}}) = \int_0^{\delta}\sqrt{1+\log N_B(\varepsilon, \mathcal{L}, \|\cdot\|) \, d\varepsilon} \leq \int_0^{\delta}1 + \left(\frac{\kappa}{\varepsilon}\right)^{1/2} \, d\varepsilon = \delta + 2\kappa^{1/2}\delta^{1/2}.
\]
Lemma 3.4.2 of \citet{vanderVaartWellner1996} with $M = 2$ implies
\begin{align*}
    & \mathbb{E}\left[\Big\Vert\mathbb{G}_n \int_{\mathbb{R}}(1\{y \leq t\} - F(\alpha^{\top}x, t))^2 dQ(t) - \int_{\mathbb{R}}(1\{y \leq t\} - F_0(\alpha_0^{\top}x, t))^2 dQ(t)  \Big\Vert_{\mathcal{L}_{\delta}}\right] \\
    & \leq (\delta + 2\kappa^{1/2}\delta^{1/2})\left(1+2\frac{\delta + 2\kappa^{1/2}\delta^{1/2}}{\delta^2n^{1/2}}\right).
\end{align*}
Consequently, with
\begin{align*}
    \tilde{\phi}_n(\delta) & := (\delta + 2\kappa^{1/2}\delta^{1/2})\left(1+2\frac{\delta + 2\kappa^{1/2}\delta^{1/2}}{\delta^2n^{1/2}}\right) \\
    \phi_n(\delta) & := \tilde{\phi}_n(\delta) / \phi_n(1),
\end{align*}
we have
\begin{align*}
    \mathbb{E}\left[\sup_{(F,\alpha)\colon d((F,\alpha), (F_0, \alpha_0)) \leq \delta} |(\mathbb{M}_n - \mathbb{M})(F,\alpha) - (\mathbb{M}_n-\mathbb{M})(F_0, \alpha_0)| \right] \leq \phi_n(\delta)
\end{align*}
and, for $r_n = n^{2/3}$, $r_n^2\phi_n(1/r_n) \leq n^{1/2}$. Since $(\hat{F}_n, \hat{\alpha}_n)$ maximizes $\mathbb{M}_n$ by definition, Theorem 3.4.2 of \citet{vanderVaartWellner1996} implies that $n^{1/3}d((\hat{F}_n, \hat{\alpha}_n), (F_0, \alpha_0)) = O_p(1)$.
\end{proof}

For the entropy of the function class
\[
    \mathcal{M} = \left\{h\colon \mathbb{R}^d\times\mathbb{R}\rightarrow[0,1], \, h(x,y) = \int_{[y,\infty]}F(\alpha^{\top}x, t) dQ(t),\, (F,\alpha) \in \mathcal{F}\right\}
\]
we begin with the simpler case that $Q$ has finite support.

\begin{lemma} \label{lem:entropy_finite_supp}
Under Assumptions \ref{assp:space} and \ref{assp:measure_Q} (i), we have
\[
    \log(N_B(\varepsilon, \mathcal{M}, \|\cdot\|_{\mathbb{P}})) \leq \frac{\tilde{C} + p}{\varepsilon}, \quad \varepsilon \in (0,1),
\]
where $\tilde{C} = \tilde{C}(d, R, \bar{p}_X)$ is the constant from Proposition \ref{prop:entropy}, and $p$ is the cardinality of the finite support of $Q$.
\end{lemma}

\begin{proof}[Proof of Lemma \ref{lem:entropy_finite_supp}]
    
Recall that $Q$ puts all its mass on the points $t_1 < \ldots < t_p$. Let $l_i, u_i$, $i = 1, \dots, N$ be $\varepsilon$-brackets covering $\mathcal{G}$, and let $l_{i(j)}, u_{i(j)}$ be an $\varepsilon$-bracket containing $x \mapsto F(\alpha^{\top}x, t_j)$, $j = 1, \dots, m$. Then,
\[
    L(x,y) := \sum_{j\colon t_j \geq y} Q(\{t_j\}) l_{i(j)}(x), \quad U(x,y) := \sum_{j\colon t_j \geq y} Q(\{t_j\}) u_{i(j)}(x)
\]
are an $\varepsilon$-bracket containing $h$, because
\begin{align*}
    \|U - L\|_{\mathbb{P}}^2 & = \int_{\mathbb{R}^d \times \mathbb{R}} \left(\sum_{j\colon t_j \geq y} Q(\{t_j\}) (u_{i(j)}(x) - l_{i(j)}(x))\right)^2 \, d\mathbb{P}(x,y) \\
    & \leq \int_{\mathbb{R}^d \times \mathbb{R}} \sum_{j\colon t_j \geq y} Q(\{t_j\}) (u_{i(j)}(x) - l_{i(j)}(x))^2 \, d\mathbb{P}(x,y) \\
    & \leq \int_{\mathbb{R}^d \times \mathbb{R}} \sum_{j=1}^p Q(\{t_j\}) (u_{i(j)}(x) - l_{i(j)}(x))^2 \, d\mathbb{P}(x,y) \\
    & =  \int_{\mathbb{R}^d} \sum_{j=1}^p Q(\{t_j\}) (u_{i(j)}(x) - l_{i(j)}(x))^2 \, d\mathbb{P}(x) \\
    & \leq \varepsilon^2.
\end{align*}
Moreover, there are $pN$ functions of the form of $L, U$, corresponding to $N$ choices for $l_{i(j)}, u_{i(j)}$ and $p$ choices of $t_j$. So for $\varepsilon \in (0,1)$, we have
\[
    \log(N_B(\varepsilon, \mathcal{M}, \|\cdot\|_{\mathbb{P}})) \leq \tilde{C}/\varepsilon + p \leq \frac{\tilde{C}+p}{\varepsilon}.
\]
\end{proof}

For $Q$ with Lebesgue continuous distribution, the entropy bound is as follows.

\begin{lemma} \label{lem:entropy_continuous_q}
Under Assumptions \ref{assp:space} and \ref{assp:measure_Q} (ii), we have
\[
    \log(N_B(\varepsilon, \mathcal{M}, \|\cdot\|_{\mathbb{P}})) \leq \frac{3\tilde{C}\max(1,c) + b - a + 1}{\varepsilon}, \quad \varepsilon \in (0,1),
\]
where $\tilde{C} = \tilde{C}(d, R, \bar{p}_X)$ is the constant from Proposition \ref{prop:entropy}.
\end{lemma}

\begin{proof}[Proof of Lemma \ref{lem:entropy_continuous_q}]
We assume that $Q$ is Lebesgue continuous on $[a,b]$ with density bounded from above by $c < \infty$. Discretize the interval $[a,b]$ with a net of suitable size, namely, let $N' = \lceil (b-a)/\varepsilon \rceil$ and define
\begin{align*}
    t_j := a + (j-1)(b-a)/N', \quad h_{j}(x) := \int_{[a,t_j]} F(\alpha^{\top}x, t) \, dQ(t), \ j = 1, \dots, N'+1.
\end{align*}
The functions $h_j$ are contained in the class $\mathcal{G}$. Let $l_i, u_i$, $i = 1, \dots, N$ be $\varepsilon$-brackets for $\mathcal{G}$, such that $N_B(\varepsilon, \mathcal{G}, \|\cdot\|_{\mathbb{P}^{X|Y=y}}) \leq \tilde{C}/\varepsilon$ for all $y \in \mathrm{supp}(\mathbb{P}^Y)$. For $j = 1, \dots, m$ let $i(j)$ be an index such that $l_{i(j)} \leq h_j \leq u_{i(j)}$, and for $t \in (-\infty, b]$, define
\[
    r(t) := \max\{j \in \{1, \dots, N'+1\}\colon t_j \leq t\}, \quad s(t) := \begin{cases}
        \min(r(t) + 1, N'+1), & \ t \geq a, \\
        r(t), & \ t < a \\
    \end{cases},
\]
and the functions
\[
    L(x,y) := l_{i(r(y))}(x), \quad U(x,y) := u_{i(s(y))}(x), \quad y \in (-\infty,b],
\]
with $L(x,y) := U(x,y) := 0$ for $y > b$. Note that there are at most $N(N'+1)$ such functions for all choices of $r(y) \in \{1, \dots, N'+1\}$ and $i(j) \in \{1, \dots, N\}$, $j = 1, \dots, N'+1$. By construction, we have 
\[
    L(x,y) \leq \int_{[y, \infty]} F(\alpha^{\top}x, t) \, dQ(t) \leq U(x,y), \ y \in \mathbb{R}.
\]
We show that $L, U$ form an $\varepsilon$-bracket. First, notice that
\begin{align*}
    \|U-L\|_{\mathbb{P}}^2 & = \int_{\mathbb{R}^d\times\mathbb{R}} (U(x,y) - L(x,y))^2 \, d\mathbb{P}(x,y) \\
    & = \int_{\mathbb{R}^d\times\mathbb{R}} (u_{i(s(y))}(x) - l_{i(r(y))}(x))^2 \, d\mathbb{P}(x,y) \\
    & = \int_{\mathbb{R}}\int_{\mathbb{R}^d} (u_{i(s(y))}(x) - l_{i(r(y))}(x))^2 \, d\mathbb{P}^{X|Y=y}(x) d\mathbb{P}^Y(y).
\end{align*}
We separate the outer integral into three parts. The lower part, over $(-\infty,a)$, satisfies
\begin{align*}
\int_{-\infty}^a \int_{\mathbb{R}^d} (u_{s(y)}(x) - l_{r(y)}(x))^2 \, d\mathbb{P}^{X|Y=y}(x) \, d\mathbb{P}^Y(y) & = \int_{-\infty}^a \int_{-\infty}^a (u_{i(1)}(x) - l_{i(1)}(x))^2 \, d\mathbb{P}^{X|Y=y}(y) \, d\mathbb{P}^Y(x) \\
& \leq \int_{-\infty}^a (u_{i(1)}(x) - l_{i(1)}(x))^2 \, d\mathbb{P}^{X|Y=y}(y),
\end{align*}
since $l_{i(1)},u_{i(1)}$ are $\varepsilon$-brackets.
The upper part over $(b,\infty)$ equals $0$ because $L(x,y) = U(x,y) = 0$ for $y > b$. For the middle part over $[a,b]$, let $y$ in $[t_j, t_{j+1})$. Then,
\begin{align*}
   \int_{\mathbb{R}^d} (u_{i(s(y))}(x) - l_{i(r(y))}(x))^2 \, d\mathbb{P}^{X|Y=y}(x) & = \int_{\mathbb{R}^d} (u_{i(j+1)}(x) - l_{i(j)}(x))^2 \, d\mathbb{P}^{X|Y=y}(x),
\end{align*}
and we expand the integrand as follows
\begin{align} \label{eq:expand_the_square}
    & \int_{\mathbb{R}^d} (u_{i(j+1)}(x) - l_{i(j)}(x))^2 \,d\mathbb{P}^{X|Y=y}(x) \nonumber\\
    & = \int_{\mathbb{R}^d} (u_{i(j+1)}(x) - h_{j+1}(x) + h_{j+1}(x) - h_j(x) + h_j(x) - l_{i(j)}(x))^2 \,d\mathbb{P}^{X|Y=y}(x).
\end{align}
Since $l(i(k))$, $i = 1, \dots, N$ are $\varepsilon$-brackets, we have
\[
    \int_{\mathbb{R}^d} (u_{i(j+1)}(x) - h_{j+1}(x))^2 + (h_j(x) - l_{i(j)}(x))^2 \,d\mathbb{P}^{X|Y=y}(x) \leq 2\varepsilon^2,
\]
and also, because $t_{j+1}-t_j \leq (b-a)/\varepsilon$
\begin{align*}
    \int_{\mathbb{R}^d} (h_{j+1}(x) - h_j(x))^2 \,d\mathbb{P}^{X|Y=y}(x) & =     \int_{\mathbb{R}^d} \left(\int_{(t_j, t_{j+1}]} F(\alpha^{\top}x, t) \, dQ(t)\right)^2 \,d\mathbb{P}^{X|Y=y}(x) \\
    & \leq \int_{\mathbb{R}^d} \left(\int_{(t_j, t_{j+1}]} 1 \, dQ(t)\right)^2 \,d\mathbb{P}^{X|Y=y}(x) \\
    & \leq \int_{\mathbb{R}^d} (c\varepsilon)^2 \,d\mathbb{P}^{X|Y=y}(x) \\
    & \leq (c\varepsilon)^2. 
\end{align*}
The cross-terms can be bounded by applying the Cauchy-Schwarz inequality,
\begin{align*}
    & \int_{\mathbb{R}^d} (u_{i(j+1)}(x) - h_{j+1}(x))(h_{j+1}(x) - h_j(x)) \,d\mathbb{P}^{X|Y=y}(x) \\
    & \leq \left(\int_{\mathbb{R}^d} (u_{i(j+1)}(x) - h_{j+1}(x))^2d\mathbb{P}^{X|Y=y}(x)\right)^{1/2} \quad \left(\int_{\mathbb{R}^d} h_{j+1}(x) - h_j(x))^2d\mathbb{P}^{X|Y=y}(x)\right)^{1/2} \\
    & \leq \max(1, c)\varepsilon^2,
\end{align*}
applying the bounds from above; the other cross terms are bounded in an analogous way. Hence,
\begin{align*}
    \int_{[a,b)} \int_{\mathbb{R}^d} (u_{s(y)}(x) - l_{r(y)}(x))^2 \,d\mathbb{P}^{X|Y=y}d\mathbb{P}^Y & = \sum_{j=1}^{N'}\int_{[t_j, t_{j+1})} \int_{\mathbb{R}^d} (u_{s(y)}(x) - l_{r(y)}(x))^2 \,d\mathbb{P}^{X|Y=y}d\mathbb{P}^Y \\
    & \leq \sum_{j=1}^{N'}\int_{[t_j, t_{j+1})} 9\max(1,c^2)\varepsilon^2 \,d\mathbb{P}^Y \\
    & \leq 9\max(1,c^2)\varepsilon^2,
\end{align*}
where the factor $9$ is due to the fact that one obtains $3$ square terms and $6$ cross-terms from expanding the square in \eqref{eq:expand_the_square}. So we have
\[
    \int_{\mathbb{R}} \int_{\mathbb{R}^d} (u_{s(y)}(x) - l_{r(y)}(x))^2 \leq 9\max(1,c^2)\varepsilon^2.
\]
Consequently, we obtain
\[
    \log(N_B(\varepsilon, \mathcal{M}, \|\cdot\|_{\mathbb{P}})) \leq \frac{3\tilde{C}\max(1,c)}{\varepsilon} + \frac{b-a+1}{\varepsilon}
\]
\end{proof}

\subsection{Proof of Proposition \ref{prop:entropy}}

\begin{proof}
Fix $\varepsilon \in (0, 1)$.  By Lemma 21 of \citet{FeigeSchechtman2002}, we know that $\mathcal{S}_{d-1}$ can be partitioned into $N$ subsets of equal size with diameter at most $\varepsilon$ such that $N \le (A/\varepsilon^2)^d$, for a universal constant $A$. Let $\alpha_1, \ldots, \alpha_N$ be points in these $N$ subsets. Furthermore, from Theorem 2.7.5 of \citet{vanderVaartWellner1996}, we can find $N' \le \exp(K/\varepsilon)$ brackets $[h^L_i, h^U_i], i =1, \ldots, N'$ with respect to the norm $\Vert \cdot \Vert_{[-2R, 2R]}$.

Let $g \in \mathcal{G}$.  Then,  $g(x)  =  h(\alpha^{\top} x)$ for some $\alpha \in \mathcal{S}_{d-1}$ and $h \in \mathcal{H}$.  Let $i \in \{1, \ldots, N\}$ and $j \in \{1, \ldots, N'\}$ such that $\Vert \alpha - \alpha_i \Vert \le \varepsilon^2$  and $  h^L_j  \le  h \le h^U_j$.  Now, it follows from the Cauchy-Schwarz inequality that 
\begin{eqnarray*}
	\alpha^{\top} x =  (\alpha - \alpha_i)^{\top} x  + \alpha^{\top}_i x  \in [ \alpha^{\top}_i x - \varepsilon^2 R,   \alpha^{\top}_i x +\varepsilon^2 R ]  \subset [-2R, 2R]
\end{eqnarray*}
By monotonicity of $h$ this implies that
\begin{eqnarray*}
	h( \alpha^{\top}_i x + \varepsilon^2 R)  \le h(\alpha^{\top} x)  \le   h( \alpha^{\top}_i x - \varepsilon^2 R)
\end{eqnarray*}  
and hence
\begin{eqnarray}\label{Brack}
	h^L_j( \alpha^{\top}_i x + \varepsilon^2 R)  \le h(\alpha^{\top} x)  \le   h^U_j( \alpha^{\top}_i x - \varepsilon^2 R).
\end{eqnarray}  
Now, using the Minkowski inequality, we have that
\begin{eqnarray*}
	\left(\int_{\mathcal{X}}  \left \{ h^U_j( \alpha^{\top}_i x - \varepsilon^2 R)  -  h^L_j( \alpha^{\top}_i x + \varepsilon^2 R)  \right \}^2   \,dx  \right)^{1/2}  & \le   & \left(\int_{\mathcal{X}}  \left \{ h^U_j( \alpha^{\top}_i x - \varepsilon^2 R)  -  h( \alpha^{\top}_i x - \varepsilon^2 R)  \right \}^2   \,dx  \right)^{1/2} \\
	&&  +   \left(\int_{\mathcal{X}}  \left \{ h( \alpha^{\top}_i x - \varepsilon^2 R)  -  h( \alpha^{\top}_i x + \varepsilon^2 R)  \right \}^2   \,dx  \right)^{1/2} \\
	&& +  \left(\int_{\mathcal{X}}  \left \{ h^L_j( \alpha^{\top}_i x + \varepsilon^2 R)  -  h( \alpha^{\top}_i x + \varepsilon^2 R)  \right \}^2   \,dx  \right)^{1/2} \\
	& =: &  I_1 + I_2 + I_3.
\end{eqnarray*}
Note that for any $\alpha=(\alpha^{(1)}, \ldots, \alpha^{(d)})  \in \mathcal{S}_{d-1}$, there exists $j \in \{1, \ldots, d \}$ such that $\vert \alpha^{(j)}  \vert \ge 1/\sqrt{d}$. Without loss of generality we assume that $\vert \alpha^{(1)}_i \vert \ge 1/\sqrt{d}$. Consider the change of variable  $\varphi(x)  =  t$ where
\begin{eqnarray*}
	t_1 =  \alpha^{\top}_i  x  -  \varepsilon^2 R \  \  \textrm{and}  \  \ t_j =  x_j,  \  \ \textrm{for  $j=2, \ldots, d$}.
\end{eqnarray*}
Then, 
\begin{eqnarray*}
	I_1  & \le   &  \left(\int_{\varphi(\mathcal{X})}   \left\{  h^U_j(t_1)  -  h(t_1) \right \}^2  dt \frac{1}{\alpha^{(1)}_j}  \right)^{1/2}   \\
	& \le  &  \left(\sqrt{d}  \int_{-2R}^{R}  \int_{-R}^R  \ldots  \int_{-R}^R \left\{  h^U_j(t_1)  -  h(t_1) \right \}^2  dt\right)^{1/2}  \\
	& \le &  d^{1/4} (2R)^{(d-1)/2}  \left(\int_{-2R}^{R}  \left \{h^U_j(t_1)  -  h(t_1) \right \}^2  dt_1\right)^{1/2}  \\
	& \le &  d^{1/4} (2R)^{(d-1)/2}  \left(\int_{-2R}^{2R}  \left \{h^U_j(t_1)  -  h(t_1) \right \}^2  dt_1\right)^{1/2}  \\
	& \le &   d^{1/4} (2R)^{(d-1)/2}  \varepsilon,
\end{eqnarray*}
where above used that   $t_1 = \alpha_j^{\top} x - \varepsilon^2 R  \in [- 2R, R]$ for all  $x \in \mathcal{X}$. Using a similar reasoning, we can bound $I_3$ by the same constant. Now, we turn to $I_2$.  With the same change of variable, we have that
\begin{eqnarray*}
	\left(\int_{\mathcal{X}}  \left \{ h( \alpha^{\top}_i x - \varepsilon^2 R)  -  h( \alpha^{\top}_i x + \varepsilon^2 R)  \right \}^2   \,dx  \right)^{1/2} & \le   &  d^{1/4} (2R)^{(d-1)/2}   \left(  \int_{-2R}^{R}    \left\{  h(z)   -   h(z +  2  \varepsilon^2 R) \right \}^2 dz \right)^{1/2}  \\
	& \le & d^{1/4} (2R)^{(d-1)/2}   \left(  \int_{-2R}^{R}    \left\{   h(z) -  h(z +  2  \varepsilon^2 R) \right \} dz \right)^{1/2},
\end{eqnarray*}
using monotonicity of $h$ and that $h(z) - h(z+2\varepsilon^2R) \in [0,1]$ for all $z \in [-2R, R]$. Now,
\begin{eqnarray*}
	\int_{-2R}^{R}    \left\{   h(z) -  h(z +  2  \varepsilon^2 R) \right \} dz &  = &   \int_{-2R}^R  h(z) dz -  \int_{-2R}^{R}  h(z +  2  \varepsilon^2 R) dz \\
	& =&   \int_{-2R}^R  h(z) dz -  \int_{-2R + 2  \varepsilon^2 R}^{R+ 2  \varepsilon^2 R}  h(z) dz  \\
	& = &    \int_{-2R}^{-2R + 2  \varepsilon^2 R}  h(z)  dz    -   \int_{R}^{R  +  2  \varepsilon^2 R}  h(z) dz  \\
	& \le &  2 \varepsilon^2 R.
\end{eqnarray*}
Thus,
\begin{eqnarray*}
	\left(\int_{\mathcal{X}}  \left \{ h^U_j( \alpha^{\top}_i x - \varepsilon^2 R)  -  h^L_j( \alpha^{\top}_i x + \varepsilon^2 R)  \right \}^2   \,dx  \right)^{1/2} & \le  &  2  d^{1/4} (2R)^{(d-1)/2}  \varepsilon   +  d^{1/4} (2R)^{(d-1)/2}  \sqrt{2} \sqrt{R}  \varepsilon 
	\\
	& \le &  2 d^{1/4} (2R)^{(d-1)/2} (1 +  \sqrt{R})  \varepsilon.
\end{eqnarray*}
If we put $B = 2 d^{1/4} (2R)^{(d-1)/2} (1 +  \sqrt{R}) $, then the previous calculations and the inequality (\ref{Brack}) imply that  
\begin{eqnarray*}
	N_B(B \varepsilon, \mathcal{G}, \Vert \cdot \Vert_{\mathcal{X}})   \le   N N'
\end{eqnarray*} 
and hence
\begin{eqnarray*}
	\log\left(N_B(B \varepsilon, \mathcal{G}, \Vert \cdot  \Vert_{\mathcal{X}}) \right)  &\le  &   \log N  +  \log N'  \\
	& \le   &   d \log \frac{A}{\varepsilon^2}   +   \frac{2 K \sqrt{R}}{\varepsilon}   \\
	& =  &   2 d \log \frac{\sqrt{A}}{\varepsilon}   +   \frac{2 K \sqrt{R}}{\varepsilon}  \\
	& \le  &   \frac{2(d \sqrt{A} + K \sqrt{R})}{\varepsilon}
\end{eqnarray*} 
which in turn implies that
\begin{eqnarray*}
	\log\left(N_B(\varepsilon, \mathcal{G}, \Vert \cdot  \Vert_{\mathcal{X}}\right)  &\le  &   \frac{ 2^{(d+1)/2} d^{1/4} R^{(d-1)/2} (1 +  \sqrt{R}) (d \sqrt{A} + 2 K \sqrt{R})}{\varepsilon}.
\end{eqnarray*}
Finally, since the Lebesgue density of $\mu$ is bounded from above by $C$, the previous bound implies
\begin{eqnarray*}
	\log\left(N_B(\varepsilon, \mathcal{G}, \Vert \cdot  \Vert_{\mu}\right)  &\le  &   \frac{ 2^{(d+1)/2} d^{1/4} R^{(d-1)/2} (1 +  \sqrt{R}) (d \sqrt{A} + 2 K \sqrt{R})  C}{\varepsilon}.
\end{eqnarray*}
\end{proof}

\subsection{Proof of Theorem \ref{thm:separated_consistency}}
\begin{proof}
For simplicity of notation, index the subsequence by $n$, and choose an $\omega$ in the underlying probability space such that (\ref{eq:asconsistencyhard}) holds true. Recall that $\hat{F}_{n}$ is non increasing in the first entry and non decreasing in the second entry for every $n$. 
Lemma 2.5. in \cite{vanderVaart1998} can be adapted to this case. Therefore $\hat{F}_{n}$ converges pointwise along a subsequence to a bivariate function $G$ at each point of continuity of $G$ that lies in $\mathrm{supp}(Q)$. The limit $G$ has the property that $G(\cdot,t)$  is left continuous and non increasing for each $t \in \mathrm{supp}(Q)$ and $G(z, \cdot)$ non decreasing for every $z$. Furthermore, $\hat{\alpha}_n \in \mathcal{S}_{d-1}$ is a sequence in a compact space and hence converges along a further subsequence to $\beta_0$ in the Eudlidean distance. 

Our goal is to show that $G = F_0$  and $\alpha_0 = \beta_0$. Recall that if the $L_2$ distance between two functions is zero then they coincide almost surely. We have
\begin{align*}
    & \int_{\mathcal{X} \times \R} (G(\beta_0^{\top}x,t) - F_0(\alpha_0^{\top}x,t))^2 d\mathbb{P}^X(x)dQ(t) \\
&  = \int_{\mathcal{X} \times \R} \Big( G(\beta_0^{\top}x,t) - G(\hat{\alpha}_n^{\top}x,t)  + \hat{F}_n(\hat{\alpha}_n^{\top}x,t) -  F_0(\alpha_0^{\top}x,t) + G(\hat{\alpha}_n^{\top}x,t) -\hat{F}_n(\hat{\alpha}_n^{\top}x,t) \Big)^2 d\mathbb{P}^X(x)dQ(t) \\
&  \leq 3I_{n,1} + 3I_{n,2} + 3I_{n,3}
\end{align*}
by applying the Cauchy-Schwarz inequality, where 
\begin{align*}
	I_{n,1} &= \int_{\mathcal{X} \times \R} \left( G(\beta_0^{\top}x,t) - G(\hat{\alpha}_n^{\top}x,t) \right)^2 d\mathbb{P}^X(x)dQ(t),  \\
I_{n,2} &= \int_{\mathcal{X} \times \R} \left(  \hat{F}_n(\hat{\alpha}_n^{\top}x,t) -  F_0(\alpha_0^{\top}x,t)    \right)^2 d\mathbb{P}^X(x)dQ(t),\\
I_{n,3} &= \int_{\mathcal{X} \times \R} \left( G(\hat{\alpha}_n^{\top}x,t) -\hat{F}_n(\hat{\alpha}_n^{\top}x,t)  \right)^2 d\mathbb{P}^X(x)dQ(t).
\end{align*}
We show that for $n \to \infty$ the terms $I_{n,1}, \ I_{n,2}, \ I_{n,3}$ converge to zero almost surely, so  $G = F_0$ almost surely.

Recall that $\hat{\alpha}_n$ converges to $\beta_0$. Therefore, at all continuity points of $G_0$ we have that $G_0(\hat{\alpha}_n^{\top}x,t) $ converges to $G_0(\beta_0^{\top}x,t)$. Note that $G_0$ is bounded and monotone in both variables. \cite{Lavric1993} shows that the set of all discontinuity points of the bivariate, monotone function $G$  may not be countable but has Lebesgue measure $0$. When using that both $Q$ and $\mathbb{P}^X$ are equivalent to the Lebesgue measure, under our assumptions, we have that $I_{n,1} \to 0$ by Lebesgue's dominated convergence Theorem. The second integral $I_{n,2}$ converges to $0$ directly by (\ref{eq:asconsistencyhard}). Finally, we rewrite the third integral to
\begin{equation*}
	I_{n,3} = \int_{\mathcal{X} \times \R} \left( G(z,t) -\hat{F}_n(z,t)  \right)^2 d\mathbb{Q}_n(z)dQ(t)
\end{equation*}
where $\mathbb{Q}_n$ denotes the distribution of $\hat{\alpha}_n^{\top}X$ and $X$ is a random variable that is independent of the data, but has distribution $\mathbb{P}^X$. As at each point of continuity of $G$,  the function $\hat{F}_n$ converges to $G$ and the set of discontinuity points of $G$ has Lebesgue measure 0, Assumption \ref{assp:distribution_alphax} and Lebesgue's dominated convergence theorem imply that $I_{n,3} \rightarrow 0$. 

If necessary, modify $G$ to not have discontinuity points at the boundary. By Proposition \ref{prop:identifiability} it follows that $\beta_0 = \alpha_0$ and $G = F_0$ everywhere on $\mathcal{C}_{\alpha_0} \times \mathrm{supp}(Q)$. As we have found almost sure convergence along a subsequence, we follow that the statements hold true for convergence in probability. 
\end{proof}

\subsection{Proof of Theorem \ref{thm:separated_rate}}

\begin{proof}	
We apply Lemma 2.5. from \cite{MurphyVanderVaartWellner1999}. Rewrite the integrated error as follows,
\begin{align*}
		\int_{\mathcal{X}\times \R}  \left( \hat{F}_{n;\hat{\alpha}_n}(\hat{\alpha}_n^{\top}x,t) - F_0(\alpha_0^{\top}x,t) \right)^2 d\mathbb{P}^X(x) dQ(t)
& = \int_{\mathcal{X}\times \R}  \left(G_1(x,t) + G_2(x,t)\right)^2 d\mathbb{P}^X(x) dQ(t) \\
& = \mathbb{E}\left[( G_1(X,T) + G_2(X,T))^2 \right] 
\end{align*}
where the expectation is a shorthand notation of integrating with respect to a random variable $(X,T)$ whose distribution is the product measure of $\mathbb{P}^X$ and $Q$. The functions $G_1$ and $G_2$ are $G_1(x,t) = \hat{F}_{n}(\hat{\alpha}_n^{\top}x,t) - F_0(\hat{\alpha}_n^{\top}x,t) = \Tilde{G}_1(\hat{\alpha}_n^{\top}x,t)$ and $G_2(x,t) = F_0( \hat{\alpha}_n^{\top}x,t) - F_0(\alpha_0^{\top}x,t)$. The Cauchy-Schwarz inequality and the tower property of conditional expectations yield
\begin{align*}
	\mathbb{E}\left[G_1(X,T) G_2(X,T) \right]^2  
&= \mathbb{E}\left[\Tilde{G}_1(\hat{\alpha}_n^{\top}X,T) G_2(X,T) \right]^2 \\
&= \mathbb{E}\left[\Tilde{G}_1(\hat{\alpha}_n^{\top}X,T)\ \mathbb{E}[ G_2(X,T) \mid \hat{\alpha}_n^{\top}X,T]\right]^2 \\
& \leq \mathbb{E}\left[\Tilde{G}_1(\hat{\alpha}_n^{\top}X,T)^2 \right] \mathbb{E} \left[ \mathbb{E}[ G_2(X,T) | \ \hat{\alpha}_n^{\top}X,T]^2 \right] \\
& = c_n \mathbb{E}\left[G_1(X,T)^2 \right] \mathbb{E}\left[G_2(X,T)^2 \right],
\end{align*}
where 
\begin{equation*}
	c_n = \frac{\mathbb{E} \left[ \mathbb{E}[ G_2(X,T) | \ \hat{\alpha}_n^{\top}X,T]^2 \right]}{\mathbb{E}\Big[G_2(X,T)^2 \Big]}
= \frac{\mathbb{E} \left[( F_0( \hat{\alpha}_n^{\top}X,T) - \mathbb{E}[ F_0(\alpha_0^{\top}X,T) | \ \hat{\alpha}_n^{\top}X,T])^2 \right]}{\mathbb{E}\Big[( F_0( \hat{\alpha}_n^{\top}X,T) - F_0(\alpha_0^{\top}X,T))^2 \Big]}.
\end{equation*}

If $c_n <1 $ it follows by \cite{MurphyVanderVaartWellner1999} that 
\begin{align}
    & \int_{\mathcal{X} \times \R}  \Big( \hat{F}_{n}(\hat{\alpha}_n^{\top}x,t) - F_0(\alpha_0^{\top}x,t) \Big)^2 d\mathbb{P}^X(x) dQ(t) \nonumber \\
     & \geq (1 - \sqrt{c_n}) \bigg( \mathbb{E}\left[(\hat{F}_{n}(\hat{\alpha}_n^{\top}X,T) - F_0(\hat{\alpha}_n^{\top}X,T))^2 \right]
    + \mathbb{E}\left[(F_0( \hat{\alpha}_n^{\top}X,T) - F_0(\alpha_0^{\top}X,T))^2 \right] \bigg). \label{eq:int_lwr_bound}
\end{align}

We now prove that there exists a $c<1$ such that from any subsequence $(n_k)_{k\in\mathbb{N}}$, there exists a subsequence $(n_{k_l})_{l\in\mathbb{N}}$ along which $\limsup_{l \to \infty} c_{n_l} \leq c <1$ almost surely. This shows that $(1-\sqrt{c_n})^{-1} = O_P(1)$. 

To prove the claim, consider an arbitrary subsequence. For simplicity of notation, index it with $n$. Define $u_n = \lVert \hat{\alpha}_n - \alpha_0 \rVert$ and $\gamma_n =  (\hat{\alpha}_n - \alpha_0)/u_n $. As $\lVert \gamma_n \rVert = 1$ and $\mathcal{S}_{d-1} $ is compact, $\gamma_n$ converges to some $\gamma_0 \in \mathcal{S}_{d-1}$ along a subsequence. Recall that $\hat{\alpha}_n$ converges to $\alpha_0$ in probability. Therefore, we can extract a further subsequence along which the convergence from $\hat{\alpha}_n$  to $\alpha_0$ and from $\gamma_n$ to $\gamma_0$ happens almost surely. To make notation less cumbersome we index this subsequence again by $n$. Fix an event $\omega$ in the underlying probability space such that $\hat{\alpha}_0 \to \alpha_0$ and $\gamma_n \to \gamma_0$, so that we can consider $\hat{\alpha}_n$ and $\gamma_n$ as non-random. 

By Assumption \ref{assp:differentiability}, for every $t \in \R$  the map $F_0(\cdot,t)$ is continuously differentiable on $\mathcal{C}_{\alpha_0}$. Extend the function $F_0(\cdot,t)$ such that it is bounded and continuously differentiable on $\R$ and the partial derivative $z \mapsto F_0^{(1)}(z,t)$ is bounded on $\R^2$. By Taylor's Theorem we have that for $x \in \mathcal{X}$ and $t \in \R$,
\begin{equation}
\label{eq:taylorhard}
    F_0(\alpha_0^{\top}x,t) = F_0(\hat{\alpha}_n^{\top}x,t) + F_0^{(1)}(\hat{\alpha}_n^{\top}x,t)(\alpha_0 - \hat{\alpha}_n)^{\top}x + o(u_n).
\end{equation}
Thus the numerator of $c_n$ becomes
\begin{align*}
    & \mathbb{E} \left[ \mathbb{E}[ F_0( \hat{\alpha}_n^{\top}X,T) - F_0(\alpha_0^{\top}X,T) | \ \hat{\alpha}_n^{\top}X,T]^2 \right] \\
    & \ \ = \mathbb{E} \left[ \mathbb{E}[ F_0^{(1)}(\hat{\alpha}_n^{\top}X,T)(\alpha_0 - \hat{\alpha}_n)^{\top}X + o(u_n) | \ \hat{\alpha}_n^{\top}X,T]^2 \right] \\
    & \ \ = \mathbb{E} \left[ \mathbb{E}[ F_0^{(1)}(\hat{\alpha}_n^{\top}X,T)(\alpha_0 - \hat{\alpha}_n)^{\top}X | \ \hat{\alpha}_n^{\top}X,T]^2 \right] + o(u_n^2) \\
\end{align*}
as the mixed term can be controlled by
\begin{equation*}
    2 o(u_n) \ \Big|  \mathbb{E} \left[ F_0^{(1)}(\hat{\alpha}_n^{\top}X,T)(\alpha_0 - \hat{\alpha}_n)^{\top}X \right] \Big| = o(u_n^2).
\end{equation*}
This is because the partial derivative $z \mapsto F_0^{(1)}(z, t)$ is bounded. Similarly the denominator becomes
\begin{equation*}
    \mathbb{E}\left[( F_0( \hat{\alpha}_n^{\top}X,T) - F_0(\alpha_0^{\top}X,T))^2 \right] = \mathbb{E}\left[ (F_0^{(1)}(\hat{\alpha}_n^{\top}X,T)(\alpha_0 - \hat{\alpha}_n)^{\top}X)^2 \right] + o(u_n^2).
\end{equation*}
We rewrite
\begin{equation*}
    c_n = \frac{\mathbb{E} \left[( F_0^{(1)}(\hat{\alpha}_n^{\top}X,T) \gamma_n^{\top}\mathbb{E}[ X | \ \hat{\alpha}_n^{\top}X,T])^2 \right] + o(1)}{\mathbb{E} \Big[( F_0^{(1)}(\hat{\alpha}_n^{\top}X,T) \gamma_n^{\top}X)^2 \Big] + o(1)}.
\end{equation*}
By Lemma 9.1 in the supplement of \cite{BalabdaouiDurotJankowski2019} we have that $\mathbb{E}[ X | \ \hat{\alpha}_n^{\top}X,T] \to \mathbb{E}[ X | \ \alpha_0^{\top}X,T]$ almost surely. By Lebesgue's dominated convergence theorem and the continuity of $F_0^{(1)}(\cdot,t)$, we have that
\begin{align*}
    \limsup_{n\to \infty} c_n &= \frac{\mathbb{E} \left[( F_0^{(1)}(\alpha_0^{\top}X,T) \gamma_0^{\top}\mathbb{E}[ X | \ \alpha_0^{\top}X,T])^2 \right] }{\mathbb{E} \Big[( F_0^{(1)}(\alpha_0^{\top}X,T) \gamma_0^{\top}X)^2 \Big]} \\
    &= \frac{\gamma_0^{\top}\mathbb{E} \left[ F_0^{(1)}(\alpha_0^{\top}X,T) ^2 \mathbb{E}[ X | \ \alpha_0^{\top}X,T]\mathbb{E}[ X | \ \alpha_0^{\top}X,T]^{\top}\right] \gamma_0}{\gamma_0^{\top} \mathbb{E} \Big[F_0^{(1)}(\alpha_0^{\top}X,T)^2 X X^{\top} \Big] \gamma_0}.
\end{align*}
As $\hat{\alpha}_n \in \mathcal{S}_{d-1}$, it follows that  $1 = \lVert \hat{\alpha}_n \rVert =  \lVert \alpha_0 + u_n\gamma_n\rVert = \lVert \alpha_0 \rVert + u_n^2 +  2 u_n \langle \alpha_0  , \gamma_n \rangle$ and thus $2\langle \alpha_0  , \gamma_n \rangle = - u_n \to 0$ and $\langle \alpha_0  , \gamma_0 \rangle = 0$. Write
\begin{equation*}
    c = \sup_{\gamma \in \mathcal{S}_{d-1} : \langle \alpha_0  , \gamma \rangle = 0 } \frac{\gamma^{\top}\mathbb{E} \left[ F_0^{(1)}(\alpha_0^{\top}X,T) ^2 \mathbb{E}[ X | \ \alpha_0^{\top}X,T]\mathbb{E}[ X | \ \alpha_0^{\top}X,T]^{\top}\right] \gamma}{\gamma^{\top} \mathbb{E} \Big[F_0^{(1)}(\alpha_0^{\top}X,T)^2 X X^{\top} \Big] \gamma}.
\end{equation*}
Then, we have that $\lim_{n\to \infty} c_n \leq c$ where $c$ does not depend on the chosen path $\omega$. It remains to prove that $c<1$. We first expand the matrix in the denominator and get
\begin{align*}
    \mathbb{E} \Big[ F_0^{(1)}(\alpha_0^{\top}X,T)^2 \ X X^{\top}  \Big] 
        &= \mathbb{E} \Big[ F_0^{(1)}(\alpha_0^{\top}X,T)^2 \mathbb{E}[X | \ \alpha_0^{\top}X,T] \ \mathbb{E}[X | \ \alpha_0^{\top}X,T]^{\top} \Big] \\
        & \ \ \ \ \ +  \mathbb{E} \Big[ F_0^{(1)}(\alpha_0^{\top}X,T)^2 (X- \mathbb{E}[X | \ \alpha_0^{\top}X,T]) (X - \mathbb{E}[X | \ \alpha_0^{\top}X,T])^{\top} \Big]  \\
        &:= A + B.
\end{align*}
Note that $\gamma_0^{\top} A \gamma_0$ equals the numerator in the expression of $c$. Consider some $\gamma \in \mathcal{S}_{d-1}$ with $ \langle \alpha_0  , \gamma \rangle = 0 $. Define the $2\times d$ matrix $A_0$ to have first row equal to $\alpha_0^{\top}$ and second row equal $\gamma^{\top}$ and $Z = (Z_1,Z_2) = A_0 X$. Since $X$ has a density that is positive on $\mathcal{X}$, the variable $Z$ admits a density that is positive on the set $\mathcal{Z} := \{A_0x\colon x \in \mathcal{X}\}$, which has non-empty interior. Then,
\begin{align*}
     & \gamma^{\top} \mathbb{E} \big[F_0^{(1)}(\alpha_0^{\top}X,T)^2(X- \mathbb{E}[X | \ \alpha_0^{\top}X,T]) (X - \mathbb{E}[X | \ \alpha_0^{\top}X,T])^{\top} \big] \gamma \\
       & \quad = \mathbb{E} \big[F_0^{(1)}(\alpha_0^{\top}X,T)^2(\gamma^{\top}X- \mathbb{E}[\gamma^{\top}X | \ \alpha_0^{\top}X])^2 \big]  
\end{align*}
is equal to zero if and only if $\gamma^{\top}X = \mathbb{E}[\gamma^{\top} X| \ \alpha_0^{\top}X,T]$ almost surely or equivalently $Z_2 = \mathbb{E}[Z_2 | Z_1]$ almost surely. This would mean that the distribution of $Z$ is concentrated on a one-dimensional subspace. This contradicts the fact that the density of $Z$ with respect to the Lebesgue measure is positive on $\mathcal{Z}$. It follows that $\gamma^{\top}B\gamma > 0$ and thus $c<1$. This proves the claim. In integral notation, it follows from \eqref{eq:int_lwr_bound} that
\begin{align*}
	& \int_{\mathcal{X} \times \R}  \left( \hat{F}_n(\hat{\alpha}_n^{\top}x,t) - F_0(\alpha_0^{\top}x,t) \right)^2 d\mathbb{P}^X(x) dQ(t)  \\
& \geq (1 - \sqrt{c_n}) \Bigg( \int_{\mathcal{X} \times \R}(\hat{F}_n(\hat{\alpha}_n^{\top}x,t) - F_0(\hat{\alpha}_n^{\top}x,t))^2 d\mathbb{P}^X dQ(t) \\
& \quad + \int_{{\mathcal{X} \times \R}}(F_0( \hat{\alpha}_n^{\top}x,t) - F_0(\alpha_0^{\top}x,t))^2 d\mathbb{P}^X(x) dQ(t) \Bigg) \\
& \geq (1 - \sqrt{c_n}) \int_{\mathcal{X} \times \R}\left(F_0(\hat{\alpha}_n^{\top}x,t) - F_0(\alpha_0^{\top}x,t) \right)^2 d\mathbb{P}^X dQ(t) \\
& = (1-\sqrt{c_n}) \int_{\mathcal{X} \times \R}\left( F_0^{(1)}(\hat{\alpha}_n^{\top}x,t)(\alpha_0 - \hat{\alpha}_n)^{\top}x + o(u_n) \right)^2 d\mathbb{P}^X dQ(t) \\
& \geq c' \lVert \hat{\alpha}_n - \alpha_0 \rVert ^2 \inf_{\beta \in \mathcal{S}_{d-1}} \int_{\mathcal{X} \times \R} (\beta^{\top}x)^2 d\mathbb{P}^X(x) dQ(t),
\end{align*}
for some $c' > 0$ by the previous observations, for $n$ large enough. 
Note that the infimum above is strictly positive and achieved for some $\beta$, as the function $\beta \mapsto \int_{\mathcal{X} \times \R} (\beta^{\top}x)^2 \mathbb{P}^X(x) dQ(t)$ is continuous, $\mathcal{S}_{d-1}$ is compact and the density $p_X$ is bounded away from zero. Thus, there exists $K > 0$ such that
\begin{align*}
    \lVert \hat{\alpha}_n - \alpha_0 \rVert ^2
    \leq K \int_{\mathcal{X} \times \R} \left(\hat{F}_n(\hat{\alpha}_n^{\top}x,t) - F_0(\alpha_0^{\top}x,t) \right)^2 d\mathbb{P}^X dQ(t)
    = O_P(n^{-2/3})
\end{align*}
for large $n$ and almost surely.

We turn to the second part. Recall that the density of $\hat{\alpha}_n^{\top}X$ is bounded from below by $\munderbar{q} >0$, so
\begin{align}
\label{eq:lwr_bound_l2_both_alphahat}
    \int_\mathcal{X \times \R}  \left( \hat{F}_n(\hat{\alpha}_n^{\top}x,t) - F_0(\hat{\alpha}_n^{\top}x,t) \right)^2 d\mathbb{P}^X(x) dQ(t)
    & \geq \underline{q} \int_{C_{\hat{\alpha}_n} \times \R} \left( \hat{F}_n(z,t) - F_0(z,t) \right)^2 dz dQ(t) \\
    & \geq \underline{q} \int_\R \int_{\underline{c} +v_n}^{\overline{c} -v_n} \left( \hat{F}_n(z,t) - F_0(z,t) \right)^2 dz dQ(t)  \nonumber 
\end{align}
with probability tending to one for $n \to \infty$, using the definition of $v_n$ and that $\lVert \hat{\alpha}_n - \alpha_0 \rVert = O_P(n^{-1/3})$. The left-hand side of \eqref{eq:lwr_bound_l2_both_alphahat} can be bounded from above by
\begin{align*}
	\int_{\mathcal{X} \times \R}  \left( \hat{F}_n(\hat{\alpha}_n^{\top}x,t) - F_0(\hat{\alpha}_n^{\top}x,t) \right)^2 d\mathbb{P}^X(x) dQ(t)
& \leq 2 \int_{\mathcal{X} \times \R}  \left( \hat{F}_n(\hat{\alpha}_n^{\top}x,t) - F_0(\alpha_0^{\top}x,t) \right)^2 d\mathbb{P}^X(x) dQ(t) \\
& \quad +2 \int_\mathcal{X \times \R}  \left( F_0(\hat{\alpha}_n^{\top}x,t) - F_0(\alpha_0 ^{\top},t) \right)^2 d\mathbb{P}^X(x) dQ(t)  .
\end{align*}
The first term is bounded $O_P(n^{-2/3})$ by Theorem \ref{thm:bundled} and the seconded term can be handled due to the fact that the absolute value of the partial derivative $F_0^{(1)}(z,t)$ is bounded by $K := \sup_{t\in\mathrm{supp}(Q)} K_t$; this yields
\begin{align*}
    \int_{\mathcal{X} \times \R}  \left( F_0(\hat{\alpha}_n^{\top}x,t) - F_0(\alpha_0^{\top}x,t) \right)^2 d\mathbb{P}^X(x) dQ(t) 
& \leq K^2 \int_{\mathcal{X} \times \R} ((\alpha_0 - \hat{\alpha}_n)^{\top} x)^2 d\mathbb{P}^X(x) dQ(t)\\
& \leq K^2 R^2 \lVert { \alpha_0 - \hat{\alpha}_n } \rVert ^2 \\
& = O_P(n^{-2/3}).
\end{align*}

\end{proof}

\subsection{Identifiability} \label{sec:identifiability}
The identifiability result in this section is a direct adaptation of Theorem 5.1 of \citet{BalabdaouiDurotJankowski2019}.

\begin{proposition}
\label{prop:identifiability}
    Assume $\mathcal{X} \subset \R^d$ is convex and has at least one interior point. Furthermore, assume $X$ has a density with respect to the Lebesgue measure which is strictly positive on $\mathcal{X}$. Suppose that for each $t \in \mathrm{supp}(Q)$ the function $F_0(\cdot, t)$ is left-continuous (or right-continuous), non constant and does not have discontinuity points on the boundary of $\mathcal{C}_{\alpha_0}$.
    Then $(F_0,\alpha_0)$ is identifiable.
\end{proposition}

\begin{proof}
    We will prove the left-continuous case; the right-continuous case can be treated with the same arguments. Consider pairs $(F,\alpha), (H,\beta) \in \mathcal{F}$ having the property that for each $t \in \mathrm{supp}(Q)$, the functions $F(\cdot,t)$ on $\mathcal{C}_{\alpha}$ and $H(\cdot, t)$ are left-continuous on $\mathcal{C}_{\beta}$, non constant and do not have discontinuity points on the boundary of their domain. Assume
    \begin{equation*}
        F(\alpha^Tx,t) = H(\beta^Tx,t)
    \end{equation*}
    for $\mathbb{P}^X$ almost all $x \in \mathbb{R}^d$. Fix $t_0 \in \mathbb{R}$ and define $f = F(\cdot, t_0)$ and $h = H(\cdot, t_0)$. By assumption we have $f(\alpha^Tx) = h(\beta^Tx)$ for almost every $x \in \mathcal{X}$. As $f,h$ are left-continuous, this holds for all points in the interior of $\mathcal{X}$. If we prove $\alpha = \beta$ we can follow that $f = h$ on the interior of $\mathcal{C}_{\alpha} = \mathcal{C}_{\beta}$. As there are no discontinuity points on the boundary, $f=h$ holds everywhere on $\mathcal{C}_{\alpha}$ and finally, so $F = H$ on $\mathcal{C}_{\alpha} \times \R$.  Therefore, it suffices to show $\alpha = \beta$.

    As $\mathcal{X}$ is convex, for $L>0$ small enough we can find an open ball $B_L$ of radius $L$ contained in $\mathcal{X}$ such that $x \mapsto f(\alpha^Tx)$ is non constant and 
    \begin{equation}
    \label{eq:baseassump}
        f(\alpha^Tx) = h(\beta^Tx) 
    \end{equation} for every $x \in B_L$. Without loss of generality, we assume that $B_L$ is centered at the origin --- if necessary, replace $f(z)$ with $f(z- \alpha^Tx_0)$ and $h(z)$ with $h(z-\beta^Tx_0)$, where $x_0$ is the center of a ball with the desired properties.
    We first show $\beta \in \{\alpha, -\alpha\}$ and then $\beta \neq -\alpha$.

    Assume for a contradiction that $\beta \notin \{ \alpha, -\alpha \}$. Then $\alpha$ and $\beta$ are linearly independent and by the Cauchy-Schwarz inequality for $v = \beta -\alpha$, it holds $v^T\alpha = \beta^T\alpha - 1 < 0$ and $v^T\beta > 0$. Using the monotonicity of $f$ and $h$ it follows that 
    \begin{align*}
        f(z) = f(\alpha^T(z \alpha )) = h(\beta^T(z \alpha )) = h(\alpha^T(z \alpha) + v^T(z \alpha)) \geq h(z), \\
        h(z) = h(\beta^T(z  \beta)) = f(\alpha^T(z \beta)) = f(\beta^T(z \beta) - v^T(z \beta)) \geq f(z),
    \end{align*}
    for each $z \in [0,L)$ and so $f(z) = h(z)$ on $[0,L)$. By the same arguments one shows $f(-z) = h(-z)$ on $[0,L)$, and so $f = h$ on $(-L,L)$. Hence, for $x \in B_L$ we have
    \begin{equation}
    \label{eq:B_L}
        f(\alpha^Tx) = f(\beta^Tx).
    \end{equation}
    Since $x \mapsto f(\alpha^Tx)$ is non-constant on $B_L$, there exists a point $b \in (-L,L)$  of strict decrease, so one of the following two conditions must hold,
    \begin{align} \label{eq:identifiability_case_minus}
    & f(b) > f(b+\epsilon), \ \epsilon \in (0,L-b); \\
    \label{eq:identifiability_case_plus}
    & f(b-\epsilon) > f(b), \ \epsilon \in (0,L+b).
    \end{align}
    The ball $B_L$ can be chosen in such a way that $b \neq 0$. 
    In the case \eqref{eq:identifiability_case_minus}, if $b>0$ we can choose $\epsilon$ small enough such that for $x:= (b+\epsilon)\beta$ it holds $x \in B_L$ and $\alpha^Tx \leq b$, since $\alpha^T\beta <1$. Then, we have
    \begin{equation*}
        f(\alpha^Tx) \geq f(b) > f(b+\epsilon) = f(\beta^Tx),
    \end{equation*}
    which contradicts \eqref{eq:B_L}.
    If $b<0$  we let $x = b\alpha$ and choose $\epsilon$ sufficiently small such that $b+\epsilon <0$ and $\beta^Tx = b \beta^T \alpha \geq b+\epsilon$. Then,
    \begin{equation*}
        f(\alpha^Tx) = f(b) > f(b+\epsilon) \geq f(\beta^Tx),
    \end{equation*}
    which contradicts \eqref{eq:B_L}, again. The second case, \eqref{eq:identifiability_case_plus}, can be proven with similar ideas. Namely, if $b<0$ choose $x = (b-\epsilon)\beta $ and $\epsilon$ small enough such that $\alpha^T\beta (b-\epsilon) \geq b$. Then,
    \begin{equation*}
        f(\beta^Tx) = f(b-\epsilon) \geq f(b) \geq f(\alpha^Tx)
    \end{equation*}
    which contradicts \eqref{eq:B_L}. If $b>0$ choose $x = b\alpha$ and $\epsilon$ small enough such that $b\alpha^T\beta \leq b-\epsilon$. Then,
    \begin{equation*}
        f(\beta^Tx) \geq f(b-\epsilon) \geq f(b) = f(\alpha^Tx)
    \end{equation*}
    which contradicts \eqref{eq:B_L}. This proves $\beta \in \{-\alpha,\alpha\}$.    
    
    Finally, we assume for a contradiction that $\beta = -\alpha$. For $z\in [0,L)$ we have
    \begin{equation*}
        f(z) = f(\alpha^T (z \alpha)) = h(\beta(z \alpha)) = h(-z),
    \end{equation*}
    by \eqref{eq:baseassump}. With the same argument one shows 
    \begin{equation*}
        h(z) = h(\beta(z \beta)) = f(\alpha(z \beta)) = f(-a).
    \end{equation*}
    Thus by monotonicity of $h$ we have for $z \in [0,L)$,
    \begin{equation*}
        f(z) = h(-z) \geq h (z) = f(-z)
    \end{equation*}
    and so $f(z) = f(-z)$ on $[0,L)$. As $f$ is also non-increasing, we conclude that $f$ is constant on $(-L,L)$, a contradiction. Consequently, $\alpha = \beta$ and Proposition \ref{prop:identifiability} follows.
\end{proof}

\subsection{Proof of Lemma \ref{lem:Invariance}}
\begin{proof}
Replacing $Y_1, \dots, Y_n$ by $f(Y_1), \dots, f(Y_n)$ in \eqref{eq:optimal_F} and the fact that $1\{Y_i \leq Y_j\} = 1\{f(Y_i) \leq f(Y_j)\}$ almost surely for $i,j=1,\dots,n$ imply $L_n(\mathbb{P}_n^{Y}; \hat{F}_{n, \hat{\alpha}_n}, \hat{\alpha}_n) = L_n(\mathbb{P}_n^{f(Y)}; \tilde{F}_{n, \hat{\alpha}_n}, \hat{\alpha}_n)$, which also yields the statement about the minimizers in (i). Part (ii) holds by definition of $t_i$, $i = 1, \dots, n$, $\tilde{F}_{n,\hat{\alpha}_n}$, and $\tilde{F}_0$.
\end{proof}

\end{document}